\documentclass{article}
\usepackage{amssymb}
\usepackage{amssymb}
\usepackage{amsmath}
\usepackage{amsfonts}
\usepackage{graphicx}
\usepackage{soul, color}
\usepackage{tikz}
\usepackage{hyperref}

\newtheorem{theorem}{Theorem}

\newtheorem{corollary}[theorem]{Corollary}

\newtheorem{lemma}[theorem]{Lemma}

\newtheorem{proposition}[theorem]{Proposition}
\newtheorem{remark}[theorem]{Remark}

\newenvironment{proof}[1][Proof]{\noindent\textbf{#1.} }{\ \rule{0.5em}{0.5em}}
\input{tcilatex}

\newcommand{\ud}{\,\mathrm{d}}
\newcommand{\p}{\ensuremath{\partial}}
\newcommand{\n}{\ensuremath{\nonumber}}
\newcommand{\eps}{\ensuremath{\varepsilon}}

\newcommand{\bigO}{\mathcal{O}}

 \allowdisplaybreaks

\title{\vspace{-50pt} On Global-in-$x$ Stability of Blasius Profiles}
\author{ \Large Sameer Iyer  \footnote{\url{ssiyer@math.princeton.edu}. Department of Mathematics, Princeton University, Fine Hall, Washington Road, Princeton, NJ 08540, USA. Supported by NSF grant DMS-1802940.} }
\date{December 10, 2018}

\begin{document}

\maketitle

\begin{abstract} We characterize the well known self-similar Blasius profiles, $[\bar{u}, \bar{v}]$, as downstream attractors to solutions $[u,v]$ to the 2D, stationary Prandtl system. It was established in \cite{Serrin} that $\| u - \bar{u}\|_{L^\infty_y} \rightarrow 0$ as $x \rightarrow \infty$. Our result furthers \cite{Serrin} in the case of localized data near Blasius by establishing convergence in stronger norms and by characterizing the decay rates. Central to our analysis is a ``division estimate", in turn based on the introduction of a new quantity, $\Omega$, which is globally nonnegative precisely for Blasius solutions. Coupled with an energy cascade and a new weighted Nash-type inequality, these ingredients yield convergence of $u - \bar{u}$ and $v - \bar{v}$ at the essentially the sharpest expected rates in $W^{k,p}$ norms. 
\end{abstract}

\section{Introduction}

\hspace{5 mm} The 2D, stationary, homogeneous Prandtl equations are given by: 
\begin{align} \label{Pr.intro}
uu_x + v u_y - u_{yy} = 0, \hspace{3 mm} u_x + v_y = 0, \hspace{3 mm} (x,y) \in \mathbb{R}_+ \times \mathbb{R}_+
\end{align}

\noindent The system is typically supplemented with initial data at $\{x = 0\}$ and boundary data at $\{y =0\}$, and $y \uparrow \infty$: 
\begin{align} \label{BC}
u|_{x = 0} = u_0(y), \hspace{5 mm} [u,v]|_{y = 0} = 0, \hspace{5 mm} u|_{y \uparrow \infty} = u_E(x). 
\end{align}

\noindent For simplicity, we will take $u_E(x) = 1$, but any constant will also work. The $x$ direction is considered a time-like direction, while the $y$-direction is considered a space-like direction, and the equation (\ref{Pr.intro}) is  considered as an evolution in the $x$ variable. Correspondingly, $u_0(y)$ is called the ``initial data" and as a general matter of terminology, in this paper the words ``global" and ``local" refer to the $x$-direction. 

The following is a classical result due to Oleinik (see \cite{Oleinik}, P. 21, Theorem 2.1.1):
\begin{theorem}[\cite{Oleinik}] \label{thm.Oleinik} Assume:
\begin{align} \n
& u_0(y) > 0 \text{ for } y > 0, \\ \n
& u_0'(0) > 0, \\ \n
& u_0 \in C^\infty(y \ge 0), \\ \n
& u_0''(y) \sim y^2 \text{ near } y = 0, \\ \n
& |u_0(y) - 1| \text{ and } \p_y^k u_0(y) \text{ decay exponentially for } k \ge 1.
\end{align}

\noindent Then there exists a global solution, $[u, v]$ to (\ref{Pr.intro}) satisfying, for some $y_0, m > 0$, 
\begin{align} \label{coe.2}
&\sup_{x } \sup_{y \in (0, y_0)} |u, v, u_y, u_{yy}, u_x| \lesssim 1, \\ \label{coe.1}
&u_y(x,0)  > 0 \text{ and } u > 0.
\end{align}
\end{theorem}

Given the global existence of a solution to (\ref{Pr.intro}), the next point is to describe more precisely the asymptotics of the evolution as $x \rightarrow \infty$. In order to do, let us introduce the self-similar Blasius solutions: 
\begin{align} \label{blasius}
[\bar{u}, \bar{v}] = \Big[f'(\eta), \frac{1}{\sqrt{x + x_0}}\{ \eta f'(\eta) - f(\eta) \} \Big], \text{ where } \eta = \frac{y}{\sqrt{x+x_0}},
\end{align}

\noindent where $f$ satisfies
\begin{align} \label{blasius.ODE}
ff'' + f''' = 0, \hspace{3 mm} f'(0) = 0, \hspace{2 mm} f'(\infty) = 1, \hspace{2 mm} \frac{f(\eta)}{\eta} \xrightarrow{n \rightarrow \infty} 1.
\end{align}

\noindent Here, $x_0 > 0$ is a free parameter. The following hold: 
\begin{align*}
0 \le f' \le 1, \hspace{3 mm} f''(\eta) \ge 0, \hspace{3 mm} f''(0) > 0, \hspace{3 mm} f'''(\eta) < 0 . 
\end{align*}

We now recall the following result of Serrin's: 

\begin{theorem}[\cite{Serrin}] \label{th.serrin} Let $u$ be a solution to (\ref{Pr.intro}), (\ref{BC}) such that $\p_y u_0(y)$ is continuous. Then the following asymptotics hold
\begin{align} \label{serrin.asy}
\| u - \bar{u} \|_{L^\infty_y} \rightarrow 0 \text{ as } x \rightarrow \infty. 
\end{align}
\end{theorem}

First, let us mention that the results in \cite{Serrin} are more general than the theorem stated above in the sense that $u_{E}(x)$ in (\ref{BC}) is allowed to be have polynomial growth in $x$, whereas in the present paper we are only concerned with constant $u_E$ (which corresponds to shear flow). 

The purpose of the present work is to further Theorem \ref{th.serrin} under the assumption of small, localized perturbations of the Blasius profile.   

\begin{theorem} \label{thm.main} Fix any $0 < \eps << 1$. Let $u_0(y)$ satisfy Oleinik's conditions as stated in Theorem \ref{thm.Oleinik}. Suppose $|\p_y^l \{u_0 - 1 \} \langle y \rangle^M| \le \eps$ for some large $M$. Fix $\gamma < M$, any $\kappa > 0$ small but arbitrary. Let $[u, v]$ be the unique solution to (\ref{Pr.intro}), (\ref{BC}) with initial data $u_0(y)$. Then the following asymptotics are valid: 
\begin{align} 
\begin{aligned} \label{est.main}
&\| \p_x^\alpha \p_y^\beta \{ u - \bar{u}  \} \langle \eta \rangle^\gamma \|_{L^p_y} \lesssim C_\kappa \langle x \rangle^{-\frac{1}{2} + \frac{1}{2p}- \alpha - \frac{\beta}{2} + \kappa}, \\
&\| \p_x^\alpha \p_y^\beta \{ v - \bar{v}  \} \langle \eta \rangle^\gamma \|_{L^p_y} \lesssim C_\kappa \langle x \rangle^{-1 + \frac{1}{2p}- \alpha - \frac{\beta}{2} + \kappa}
\end{aligned}
\end{align}
\end{theorem}

The asymptotic information obtained in estimate (\ref{est.main}) is much more precise than (\ref{serrin.asy}). In fact, the estimates (\ref{est.main}) are essentially the optimal expected rates in the following sense: 

\begin{remark} \label{remark.1} Given localized initial data, $h_0(y)$, one expects solutions, $h(x,y)$ to the heat equation, $(\p_x - \p_{yy})h = 0$ to decay at rates $\| \p_x^\alpha \p_y^\beta h \|_{L^p_y} \lesssim \langle x \rangle^{-\frac{1}{2}+\frac{1}{2p} - \alpha - \frac{\beta}{2}}$.  
\end{remark}

\noindent Note that we require the small $\kappa > 0$ in (\ref{est.main}) to avoid logarithmic singularities at $x = \infty$.

One of the motivations for establishing quantitive estimates of the type (\ref{est.main}) is due to recent advances in the validity theory for steady Navier-Stokes flows, for instance the works of \cite{GI}, \cite{GIb}. In particular, using the estimates (\ref{est.main}) we can generalize the class of data treated by \cite{GI}:

\begin{corollary} Consider initial data, $u_0(y)$, that is a small perturbation of Blasius in the sense of Theorem \ref{thm.main}. Then for $x_0 >> 1$, we may take $[u(x_0, \cdot), v(x_0, \cdot)]$ as the $\{x = 0\}$ data in Theorem 1 of \cite{GI}. 
\end{corollary}
\begin{proof} This follows immediately upon applying the estimates (\ref{est.main}) above in the proof of Lemma 9 of \cite{GI}.  
\end{proof}

A second motivation for this work is that in order to prove the global validity of steady Prandtl expansions, a work currently underway by the author, one needs a precise understanding of the decay mechanism in the Prandtl equations, which is established in the present work. 
  
Let us now point the reader towards the related work of \cite{DM}, which studies the formation of singularities (in this context called ``separation") for the inhomogeneous version of (\ref{Pr.intro}) (with adverse pressure gradient).

\subsection{Notation and Main Objects}

\hspace{5 mm} We now introduce the main notations that will be in use for this paper. First, we introduce the stream function, 
\begin{align} \label{eq.psi}
\psi = \int_0^y u(x,y') \ud y'. 
\end{align}

\noindent A classical idea (\cite{Oleinik}) is to write the Prandtl system, (\ref{Pr.intro}) in the variables $(x,\psi)$: 
\begin{align*}
\p_x (u^2) - u \p_{\psi \psi} (u^2) = 0. 
\end{align*}

Define the difference unknowns: 
\begin{align*}
\phi(x,\psi) := u^2(x,\psi) - \bar{u}^2(x,\psi), \hspace{5 mm} \rho(x,\psi) := u(x,\psi) - \bar{u}(x,\psi). 
\end{align*}

\noindent It is shown in \cite{Serrin}, equation (11), that $\phi$ satisfies the equation
\begin{align} \label{nonlin.eq}
\phi_x - u \phi_{\psi \psi} + A \phi = 0, \hspace{5 mm} A = -2 \frac{\bar{u}_{yy}}{\bar{u}(\bar{u} + u)}
\end{align}

Recall the self-similar variable $\eta$ as defined in (\ref{blasius}). For simplicity, we will set the parameter $x_0 = 1$. We define a new self-similar variable, which reflects the diffusive scaling in (\ref{nonlin.eq}) via
\begin{align*}
\xi := \frac{\psi}{\sqrt{x+1}}.
\end{align*}

\noindent By using that $u \sim \eta$ for $\eta \lesssim 1$ (as established in Theorem \ref{thm.Oleinik}) and integrating equation (\ref{eq.psi}) via
\begin{align*}
\psi = \psi(x,y) = \int_0^y \bar{u}(x, y') \ud y' = \int_0^y \eta' \ud y' = \frac{1}{\sqrt{x}} \int_0^y y' \ud y' \sim y \eta = \eta^2 \sqrt{x},
\end{align*}

\noindent we obtain the relation 
\begin{align*}
\sqrt{\xi} = \eta \text{ for } \eta \le 1. 
\end{align*}

The basic object of study throughout the paper will be $\phi$, which satisfies the equation (\ref{nonlin.eq}), in the variables $(x,\psi)$ and correspondingly the self-similar variable $\xi$. 

Let us now give a brief review of the properties of $\bar{u}$ and $u$. First, Oleinik's global existence result, Theorem \ref{thm.Oleinik}, gives that $u \sim \eta$ near $\eta \le 1$. Regarding $\bar{u}$, the main properties are summarized in (\ref{blasius.ODE}). Of particular note is the concavity of $\bar{u}$, guaranteed by $f''' < 0$. In particular this implies that $A \ge 0$ in (\ref{nonlin.eq}). 

We will now introduce the norms in which we measure the solution $\phi$. First, we simplify notation throughout the paper by putting $\phi^{(k)} := \p_x^k \phi$. 
\begin{align} \label{norm.X}
&\| \phi \|_X := \sum_{k= 0}^{K_0} \| \phi \|_{X_{k}}, \\ \n
&\| \phi \|_{X_k} := \| \phi^{(k)} \langle x \rangle^{k - \sigma_k} \|_{L^\infty_x L^2_\psi} + \| \phi^{(k)} \frac{\langle \sqrt{\psi} \rangle}{\sqrt{u}} \langle x \rangle^{k - \sigma_k} \|_{L^\infty_x L^2_\psi} \\ 
& \hspace{15 mm} + \| \sqrt{u} \phi^{(k)}_\psi \langle x \rangle^{k - \sigma_k} \|_{L^2_x L^2_\psi} + \| \phi^{(k)}_\psi \sqrt{\psi} \langle x \rangle^{k - \sigma_k} \|_{L^2_x L^2_\psi}. 
\end{align}

\noindent Above, we let $\sigma_k$ be a sequence such that $\sigma_0 = 0$ and $\sigma_{k+1} > \sigma_k$. $K_0$ will be a fixed, large number. We will denote by $E_k(x)$ an arbitrary quantity satisfying $\sup_x |E_k(x)| \lesssim \sum_{j = 0}^k \| \phi \|_{X_j}$. Similarly, we will denote by $I_k(x)$ an arbitrary quantity satisfying $\int_0^\infty |I_k(x)| \ud x \lesssim \sum_{j = 0}^k \| \phi \|_{X_j}$.

By rescaling, we may arrange that the quadratic terms in (\ref{nonlin.eq}) have a factor of $\eps$ in front of them: 
\begin{align} \label{eps.nonlin}
\phi_x - (\bar{u} + \eps \rho) \phi_{\psi \psi} + A \phi = 0, \hspace{5 mm} A = -2 \frac{\bar{u}_{yy}}{\bar{u}(2\bar{u} + \eps \rho)}.
\end{align}

\subsection{Main Ideas}

\hspace{5 mm}  The main mechanisms can be summarized in four steps listed below. Overall, at each order of $x$ regularity up to $\p_x^{K_0}$ for a fixed $K_0$ large, there are two estimates that are performed. We call these the ``Energy estimate" and the ``Division estimate". This results in the control of the norm $\| \phi \|_X$ as shown above. 

\subsubsection*{\normalfont \textit{Step 1: $L^2$ level}}

\hspace{5 mm}  At the $L^2$ level, we may center our discussion around the linearized operator from (\ref{eps.nonlin}), which reads 
\begin{align} \label{lin.1}
\phi_x - \bar{u} \phi_{\psi \psi} + \bar{A} \phi, \hspace{5 mm} \bar{A} = - \frac{\bar{u}_{yy}}{\bar{u}^2}.
\end{align}

The standard energy estimate gives a bound on $\sup_x \| \phi \|_{L^2_\psi}^2 +  \|\sqrt{\bar{u}} \phi_\psi \|_{L^2_x L^2_\psi}^2$. The $\bar{A}$ term does not play a role at this point, as $\bar{A} \ge 0$ and may thus be ignored. 

The second estimate at the $L^2$ level is the ``Division estimate", which can be found in Lemma \ref{division}. There are two distinguished features of the quantities that are controlled (see the estimate (\ref{div.low})). First, there is a far-field weight of $\langle \psi \rangle$. Second, there is a nonlinear weight $\frac{1}{u}$ which gives additional control near the boundary $\{\psi = 0\}$. 

The reason we can close this Division estimate is due to the precise structure of Blasius solutions. Indeed, the choice of weight $\frac{\langle \psi \rangle}{u}$ is specially designed so that the interaction with the linearized equation, (\ref{lin.1}), produces the quantity
\begin{align*}
\int \phi^2 \times \text{positive quantities} \times \Omega, \text{ where }  \Omega = - \bar{u}_{yy} + \frac{1}{2}\bar{u} \bar{u}_x. 
\end{align*}

\noindent This type of quantity would be out of reach of the norm $X$. However, because $\bar{u}$ is a Blasius solution (not just a generic Prandtl solution), we are able to show that $\Omega(x,y)$ is globally positive. This is the content of Lemma \ref{Omega.lemma}.

The reason we need the division estimate is two-fold, corresponding to the two weights. The weight $\langle \psi \rangle$ comes in for Step 4, whereas the boundary weight $\frac{1}{u}$ comes in for Step 3. 

\subsubsection*{\normalfont \textit{Step 2: $H^k$ for $1 \le k \le K_1$}}

\hspace{5 mm}  We now fix $K_1$ so that $1 << K_1 << K_0$. The tier of derivatives between $1$ and $K_1$ we call the ``middle tier". The middle tier is distinguished from the top tier because we are able to expend derivatives. The middle tier is distinguished from the bottom ($L^2$) tier because the linearized equation is no longer (\ref{lin.1}), but rather 
\begin{align} \label{lin.2}
\phi_x - \bar{u} \phi_{\psi \psi} + \bar{A} \phi - \frac{\p_x \bar{u}}{\bar{u}} \phi. 
\end{align}

\noindent We arrive here by substituting the equation (\ref{eps.nonlin}) upon differentiating it in $x$. The reason the linearized equation has changed is due to the quasilinearity present in (\ref{lin.1}). At this stage we repeat the process of Step 1, taking advantage of the further property of Blasius solutions that $\p_x \bar{u} < 0$. 

\subsubsection*{\normalfont \textit{Step 3: $H^k$ for $K_1 + 1 \le k \le K_0$}}

\hspace{5 mm} We now arrive at the top tier of derivative in the norm $X$. The top tier is distinguished because we do not have derivatives to expend. First of all, we select $K_0$ large enough so that the ``tame principle" kicks in. For instance, terms like $\p_x^j f \times \p_x^{K_0-j} g \times \p_x^{K_0} h$ are bound to have either $j$ or $K_0 - j$ to be much smaller than $K_0 - 3$. This is standard in quasilinear problems.

The important part, however, is that the crucial weight of $\frac{1}{u}$ available due to the division estimate, is used to ``save derivatives". This is most easily seen in a term such as (``high" and ``low" refer to order of $x$ derivative):
\begin{align*}
\int \phi^{\text{high}} \phi^{\text{low}}_{\psi \psi} \phi^{\text{high}} \chi(\xi \lesssim 1). 
\end{align*}

\noindent For a term such as this, we are forced to put $\phi^{\text{low}}_{\psi \psi}$ in an $L^\infty$ type norm in order to conserve the high derivatives. To do this, with the weights of $\bar{u}$ distributed as optimally as we are allowed with the $X$ norm, we \textit{must} invoke the additional $\frac{1}{\sqrt{u}}$ weight available due to the division estimate. This is quantified by proving a localized, optimal weight, uniform estimate on $\phi^{\text{low}}_{\psi \psi}$ (see for instance, (\ref{unif.2})).

\subsubsection*{\normalfont \textit{Step 4: Optimal Decay}}

\hspace{5 mm}  Using Steps 1 - 3 we are able to show global existence of $\phi$ in the space $X$. The space $X$ certainly encodes decay information regarding the solution $\phi$ - this is evident by consulting (\ref{norm.X}). However, if one is comparing to the expected optimal asymptotics of parabolic equations in dimension one, in the sense of Remark \ref{remark.1}, then one notices that the space $X$ is a factor of $\langle x \rangle^{\frac{1}{4}}$ weaker than optimal. 

The reader should now recall the classical Nash inequality, \cite{Nash}, which states that $\| \phi \|_{L^2_\psi}^2 \lesssim \| \phi_\psi \|_{L^2_\psi}^{\frac{2}{3}} \| \phi \|_{L^1_\psi}^{\frac{4}{3}}$. Typically, one uses this by saying $\| \phi \|_{L^1_\psi}$ is conserved (say) and thus one inserts the Nash inequality to the basic energy bound to obtain an ODE of the form $\dot{\eta} + \eta^3 = 0$, for $\eta = \| \phi \|_{L^2_\psi}^2$, which immediately results in $\langle x \rangle^{-\frac{1}{4}}$ decay of $\| \phi \|_{L^2_\psi}$. 

In our case, two difficulties are present in order to carry out this procedure to optimize the decay. First, we only have the degenerate weighted quantity $\| \sqrt{\bar{u}} \phi_\psi \|_{L^2_\psi}$ appearing in the energy. Second, we cannot control $\| \phi \|_{L^1}$ by integrating the equation. 

To contend with these difficulties, we establish a new Nash-type inequality in Lemma \ref{lemma.nash} which (1) accounts for the degenerate weight of $\sqrt{\bar{u}}$ and (2) replaces the $L^1$ norm by $L^2(\sqrt{\langle \psi \rangle})$ (which scales the same way). The type of inequality we are able to establish is piecewise (as is seen from Lemma \ref{lemma.nash}). Remarkably, both upper bounds in estimate (\ref{Nash.1}) yield the same, optimal, decay rate of $\langle x \rangle^{-\frac{1}{4}}$. 

\section{Baseline Tier: $L^2$ Estimates}

In this section, we obtain two estimates at the $L^2$ level - the energy estimate and the division estimate. The reader is urged to keep in mind the linearized structure which is present at the $L^2$ level, equation (\ref{lin.1}).

\begin{lemma}[Energy Estimate] Let $\phi$ solve (\ref{eps.nonlin}). Then for $K_0 >> 1$, 
\begin{align} \label{energy.1}
\frac{\p_x}{2} \int \phi^2 + \int \bar{u} |\phi_\psi|^2 \lesssim \eps \langle x \rangle^{-(1+)} E_{K_0}(x). 
\end{align}
\end{lemma}
\begin{proof} We take inner product against $\phi$ to obtain 
\begin{align}
\frac{\p_x}{2} \int \phi^2 + \int \bar{u} |\phi_\psi|^2 - \frac{1}{2} \int \bar{u}_{\psi \psi} |\phi|^2 + \int \eps \rho \phi_{\psi \psi} \phi + \int A |\phi|^2 = 0.
\end{align}

\noindent We now use that $u \ge 0$ and $\bar{u}_{\psi \psi} \le 0$. For $\bar{u}_{\psi \psi}$, we use the relation $\bar{u}_{yy} = \bar{u}^2 \bar{u}_{\psi \psi} + \bar{u} |\bar{u}_\psi|^2$ to conclude that $\bar{u}_{\psi \psi} \le 0$ if $\bar{u}_{yy} \le 0$, which holds by properties of the Blasius profile. We also use that $A \ge 0$, which again holds by the concavity of the Blasius profile. Finally, we estimate 
\begin{align*}
|\int \eps \rho \phi_{\psi \psi} \phi \| \lesssim & \eps \langle x \rangle^{-(\frac{5}{4}-)} \Big\| \frac{\rho}{\sqrt{u}} \langle x \rangle^{\frac{1}{4}-} \Big\| \| \sqrt{u} \phi_{\psi \psi} \langle x \rangle^{1-} \|_{L^2_\psi} \| \phi \|_{L^2_\psi} \\
\lesssim & \eps \langle x \rangle^{-(\frac{5}{4}-)} E_{K_0}(x),
\end{align*}

\noindent where we have used inequalities (\ref{L2.emb}.11) and (\ref{unif.0}.4). This concludes the proof. 
\end{proof}

\begin{lemma}[Division Estimate] \label{division} Let $\phi$ solve (\ref{eps.nonlin}). For $K_0$ sufficiently large, 
\begin{align} \label{div.low}
\frac{\p_x}{2} \int \phi^2 \frac{1}{u} \langle \psi \rangle + \int \phi_\psi^2 \langle \psi \rangle \lesssim \eps \langle x \rangle^{-(1+)} E_{K_0}(x).
\end{align}
\end{lemma}
\begin{proof} We have the following identity 
\begin{align} \label{div.id.1}
\frac{\p_x}{2} \int \phi^2 \frac{1}{u} \langle \psi \rangle + \int \phi_\psi^2 \langle \psi \rangle + \int \phi^2 \langle \psi \rangle \frac{1}{2} \frac{u_x}{u^2}  + \int A \phi^2 \frac{\langle \psi \rangle}{u} = 0.
\end{align}

\noindent We group the latter two terms together via 
\begin{align*}
(\ref{div.id.1}.3) + (\ref{div.id.1}.4) =: \int \phi^2  \langle \psi \rangle \frac{1}{u^3} \mathring{\Omega} = \int \phi^2 \langle \psi \rangle \frac{1}{u^3} [\Omega + \Omega^R],
\end{align*}

\noindent where 
\begin{align*}
&\mathring{\Omega} = \frac{uu_x}{2} + u^2 A = \frac{uu_x}{2} - 2u^2 \frac{\bar{u}_{yy}}{\bar{u}(\bar{u} + u)}, \\
&\Omega = - \bar{u}_{yy} + \frac{1}{2} \bar{u} \bar{u}_x, \\
&\Omega^R := A \eps \phi + \frac{\bar{u}_{yy} \eps \rho}{2 \bar{u} + \eps \rho} + \frac{\eps \rho \bar{u}_x + \eps \bar{u} \rho_x + \eps^2 \rho \rho_x}{2}.
\end{align*}

By Lemma \ref{Omega.lemma}, the $\Omega$ contribution is positively signed. We thus need to estimate the nonlinear part in $\Omega^R$:  
\begin{align} \n
|\int \phi^2 \langle \psi \rangle \frac{1}{u^3} \Omega^R| \le & \eps \| \phi \sqrt{\langle \psi \rangle} \frac{1}{\sqrt{u}} \|_{L^2_\psi}^2 \Big[ \Big\| \frac{A \phi}{u^2} \Big\|_{L^\infty_\psi} +  \Big\| \frac{\bar{u}_{yy} \rho}{u^2(2\bar{u} + \eps \rho)} \Big\|_{L^\infty_\psi}  \\ \label{Omega.R}
& + \frac{1}{2} \Big\| \frac{\rho \bar{u}_x}{u^2} \Big\|_{L^\infty_\psi} + \frac{1}{2} \Big\| \frac{\rho_x \bar{u}}{u^2} \Big\|_{L^\infty_\psi} + \frac{\eps}{2} \Big\| \frac{\rho \rho_x}{u^2} \Big\|_{L^\infty_\psi} \Big] \\ \n
\lesssim & \eps \langle x \rangle^{-(1+)} E_{K_0}(x) \| \phi \sqrt{\psi} \frac{1}{\sqrt{u}} \|_{L^2_\psi}^2. 
\end{align}

We now proceed to prove the final inequality above after equation (\ref{Omega.R}) by estimating all five of the $L^\infty_\psi$ terms above. First, upon invoking (\ref{unif.0}.1),
\begin{align*}
\Big\| \frac{A \phi}{u^2} \Big\|_{L^\infty_\psi} \le & \| A \|_{L^\infty_\psi} \Big\| \frac{\phi}{u^2} \Big\|_{L^\infty_\psi} \lesssim  \langle x \rangle^{-1} \langle x \rangle^{-(0+)} E_{K_0}(x). 
\end{align*}

\noindent Above, we have also estimated $A$ via 
\begin{align} \n
\| A  \|_{L^\infty_\psi} \lesssim & \Big\| \frac{\bar{u}_{yy}}{\bar{u}(\bar{u} + u)} \Big\|_{L^\infty_\psi} \\ \label{est.A}
\lesssim & \Big\| \frac{\bar{u}_{yy}}{\eta^2} \chi(\eta \lesssim 1) \Big\|_{L^\infty_\psi} + \| \bar{u}_{yy} \chi(\eta \gtrsim 1) \|_{L^\infty_\psi} \lesssim \langle x \rangle^{-1}. 
\end{align}

Second, upon invoking (\ref{unif.0}.4) and that $u \sim \eta$, $\bar{u}_{yy} \sim \eta^2 \langle x \rangle^{-1}$ near $\eta = 0$, we estimate 
\begin{align*}
\Big\| \frac{\bar{u}_{yy} \rho}{u^2(2\bar{u} + \rho)} \Big\|_{L^\infty_\psi} \le \Big\| \frac{\bar{u}_{yy}}{u^2} \Big\|_{L^\infty_\psi} \Big\| \frac{\rho}{2\bar{u} + \rho} \Big\|_{L^\infty_\psi} \lesssim \langle x \rangle^{-1} \langle x \rangle^{-(0+)}E_{K_0}(x). 
\end{align*}

Third, again upon invoking (\ref{unif.0}.4), 
\begin{align*}
\Big\| \frac{\rho \bar{u}_x}{u^2} \Big\|_{L^\infty_\psi} \lesssim \Big\| \frac{\rho}{u} \Big\|_{L^\infty_\psi} \Big\| \frac{\bar{u}_x}{u} \Big\|_{L^\infty_\psi} \lesssim \langle x \rangle^{-(0+)} \langle x \rangle^{-1} E_{K_0}(x). 
\end{align*}

Fourth, again upon invoking (\ref{unif.0}.4),
\begin{align*}
\Big\| \frac{\rho_x \bar{u}}{u^2} \Big\|_{L^\infty_\psi} \lesssim \Big\| \frac{\bar{u}}{u} \Big\|_{L^\infty_\psi} \Big\| \frac{\rho_x}{u} \Big\|_{L^\infty_\psi} \lesssim \langle x \rangle^{-(1+)} E_{K_0}(x). 
\end{align*}

Fifth, again upon invoking (\ref{unif.0}.4),
\begin{align*}
\Big\| \frac{\rho \rho_x}{u^2} \Big\|_{L^\infty_\psi} \le \Big \| \frac{\rho}{u} \Big\|_{L^\infty_\psi} \Big\| \frac{\rho_x}{u} \Big\|_{L^\infty_\psi} \lesssim \langle x \rangle^{-(0+)} \langle x \rangle^{-(1+)} E_{K_0}(x). 
\end{align*}

\noindent Inserting these estimates into (\ref{Omega.R}) yields the estimate shown beneath (\ref{Omega.R}). This concludes the proof. 

\end{proof}

\section{Middle Tier: $H^k$ for $1 \le k \le K_1$}

\hspace{5 mm} At the $H^1$ level, the linearized equation changes and so requires a new treatment. Taking one $x$ derivative of (\ref{eps.nonlin}), we obtain 
\begin{align*}
\phi^{(1)}_x - u \phi^{(1)}_{\psi \psi} + A \phi^{(1)} - u^{(1)} \phi_{\psi \psi} + A_x \phi = 0. 
\end{align*}

\noindent The point here is that $u^{(1)}$ can be separated into $u^{(1)} =: \bar{u}^{(1)} + \rho^{(1)}$. While the $\rho^{(1)}$ contribution is quadratic, the $\bar{u}^{(1)}$ contribution is linear and highest order in $\phi$. To see this, we use the equation to rewrite $\phi_{\psi \psi}$ via: 
\begin{align} \label{H1.eqn}
\phi^{(1)}_x - u \phi^{(1)}_{\psi \psi} + A \phi^{(1)} - u^{(1)} \Big[ \frac{\phi^{(1)}}{u} + \frac{A\phi}{u} \Big] + A_x \phi = 0
\end{align}

\begin{lemma}[Energy Estimate] \label{L.E.E} Let $\phi$ solve the equation (\ref{eps.nonlin}). Let $0 < \delta < < 1$ be arbitrary. Then the following inequality is valid: 
\begin{align} \n
\frac{\p_x}{2} \int |\phi^{(1)}|^2  \langle x \rangle^{2-2\sigma_1} &+ \int u |\phi^{(1)}_\psi|^2 \langle x \rangle^{2-2\sigma_1} \\ \n
\le & \eps \langle x \rangle^{-(1+)} E_{K_0}(x)  + \delta \langle x \rangle^{-(1+)} \| \phi^{(1)} \langle x \rangle^{1-\sigma_1} \|_{L^2_\psi}^2\\ \label{rhs.1}
& + C_\delta \langle x \rangle^{-(1+)}E_0(x) + C_\delta \langle x \rangle^{-(0+)}I_0(x). 
\end{align}
\end{lemma}
\begin{proof} \textit{Step 1: Induction on Weights}
We use induction on the weights, which amounts to applying three successive weighted multipliers: 
\begin{align*}
\phi^{(1)}, \phi^{(1)} \langle x \rangle^{1-2\omega_1^1}, \text{ and } \phi^{(1)} \langle x \rangle^{2-2\omega_1^2},
\end{align*} 

\noindent where $\omega_1^1 < \omega_1^2 = \sigma_{1}$. Call the right-hand side of (\ref{rhs.1}) $R(x)$. One obtains the following two inequalities: 
\begin{align} \label{induct.1}
&\frac{\p_x}{2} \int |\phi^{(1)}|^2 + \int u |\phi^{(1)}_\psi|^2 \lesssim \langle x \rangle^{-(2+)} R(x), \\ \label{induct.2}
&\frac{\p_x}{2} \int |\phi^{(1)}|^2 \langle x \rangle^{1-2\omega^1_1} + \int u |\phi^{(1)}_\psi|^2 \langle x \rangle^{1-2\omega^1_1} \lesssim \langle x \rangle^{-(1+)} R(x) + \sup_x \| \phi^{(1)} \|_{L^2_\psi}^2 \langle x \rangle^{-(0+)}. 
\end{align}

\noindent Integrating from $x = \infty$ one obtains 
\begin{align} \label{induct.3}
\sup_x \int |\phi^{(1)}|^2 \langle x \rangle^{1-2\omega^1_1} \lesssim \langle x \rangle^{-(0+)} R(x). 
\end{align}

The establishment of (\ref{induct.1}) and (\ref{induct.2}) is in an identical fashion to the top order weight, so we omit it and just focus on the top order weight. We assume, thus, inductively that (\ref{induct.3}) is true.

\textit{Step 2: Top order weight:}
We apply the multiplier $\phi^{(1)} \langle x \rangle^{2-2\sigma_1}$ to (\ref{H1.eqn}). The first three terms have been treated already, with the modification that the $\p_x$ term contributes a factor of:  
\begin{align*}
\int |\phi^{(1)}|^2 \langle x \rangle^{1-2\sigma_1} \le \langle x \rangle^{-(0+)} \int |\phi^{(1)}|^2 \langle x \rangle^{1-2\omega^1_1} \lesssim \langle x \rangle^{-(0+)}R(x). 
\end{align*}

The main new leading order contribution is the fourth, which is
\begin{align} \label{split.negative}
- \int \frac{\bar{u}_x}{u} |\phi^{(1)}|^2 \langle x \rangle^{2-2\sigma_1} - \int \frac{\eps \rho_x}{u} |\phi^{(1)}|^2 \langle x \rangle^{2-2\sigma_1}. 
\end{align}

\noindent The key point is that the first term above is nonnegative because $\bar{u}_x < 0$ for Blasius solutions. 
\begin{align*}
- \int \frac{\bar{u}_x}{u} |\phi^{(1)}|^2 \langle x \rangle^{2 - 2\sigma_1} > 0. 
\end{align*}

\noindent We estimate the $\rho$ contribution from (\ref{split.negative}), which enables us to use the smallness of $\eps$: 
\begin{align*}
|(\ref{split.negative}.2)| \lesssim & \eps \Big\| \frac{\rho^{(1)}}{u} \Big\|_{L^\infty} \| \phi^{(1)} \langle x \rangle^{1-\sigma_1} \|_{L^2_\psi}^2 \lesssim  \eps \langle x \rangle^{-(1+)} \| \phi^{(1)} \langle x \rangle^{1-\sigma_1} \|_{L^2_\psi}^2 \\
\lesssim & \eps \langle x \rangle^{-(1+)} E_{K_0}(x), 
\end{align*}

\noindent where we have invoked estimate (\ref{unif.0}.4). 

We now need to estimate (\ref{H1.eqn}.5):
\begin{align} \n
|\int u^{(1)} \frac{1}{u} A &\phi \phi^{(1)} \langle x \rangle^{2-2\sigma_1}| \\ \label{pert.1}
& \le |\int \frac{\bar{u}^{(1)}}{u} A \phi \phi^{(1)} \langle x \rangle^{2-2\sigma_1}| + |\int \frac{\rho^{(1)}}{u} A \phi \phi^{(1)} \langle x \rangle^{2-2\sigma_1}|.
\end{align}

For the first term in (\ref{pert.1}), we do not have any smallness, so we take advantage of the fact that one of the terms, $\phi$, is lower order: 
\begin{align*}
|(\ref{pert.1}.1)| \lesssim & \langle x \rangle^{-(1+)} \Big\| \frac{\bar{u}^{(1)}}{u} \langle x \rangle \Big\|_{L^\infty_\psi} \| A \langle x \rangle \|_{L^\infty_\psi} \| \phi \|_{L^2_\psi} \| \phi^{(1)} \langle x \rangle^{1-\sigma_1} \|_{L^2_\psi} \\
\lesssim & C_\delta \langle x \rangle^{-(1+)} \| \phi \|_{L^2_\psi}^2 + \delta \langle x \rangle^{-(1+)} \| \phi^{(1)} \langle x \rangle^{1-\sigma_1} \|_{L^2_\psi}^2 \\
\lesssim & C_\delta \langle x \rangle^{-(1+)} E_0(x) + \delta \langle x \rangle^{-(1+)} \| \phi^{(1)} \langle x \rangle^{1-\sigma_1} \|_{L^2_\psi}^2. 
\end{align*}

\noindent Above, we have used the estimate $\bar{u} \gtrsim \eta$ and $u \gtrsim \eta$ on $\eta \lesssim 1$ and (\ref{est.A}). We have also used that $|\frac{\bar{u}_x}{\bar{u}}| \lesssim \langle x \rangle^{-1}$. 

For the second term from (\ref{pert.1}), we use the smallness of $\eps$, as this term is cubic. We treat two different cases based on the location of $\xi$. 
\begin{align*}
|(\ref{pert.1}.2)| \le |\int \frac{\rho^{(1)}}{u} A \phi \phi^{(1)} \langle x \rangle^{2-2\sigma_1} \chi(\xi) | + |\int \frac{\rho^{(1)}}{u} A \phi \phi^{(1)} \langle x \rangle^{2-2\sigma_1} \chi(\xi)^C |.
\end{align*}

First, 
\begin{align*}
|(\ref{pert.1}.2.1)| \le &\langle x \rangle^{-(\frac{1}{2}+)} \Big\| \frac{\phi}{\sqrt{u}} \langle x \rangle^{\frac{1}{4}-} \Big\|_{L^\infty_\psi(\xi \lesssim 1)} \| \rho^{(1)} u \langle x \rangle^{1-\sigma_1} \|_{L^2_\psi(\xi \lesssim 1)} \\
& \times \Big\| \frac{\phi^{(1)}}{u^{\frac{3}{2}}} \langle x \rangle^{\frac{1}{2}-} \Big\|_{L^2_\psi(\xi \lesssim 1)} \| A \langle x \rangle \|_{L^\infty_\psi} \\
\lesssim & \eps^3 \langle x \rangle^{-(0+)} I_{K_0}(x) + \eps^3 \langle x \rangle^{-(1+)} E_{K_0}(x). 
\end{align*}

\noindent Above, we have performed the Hardy-type inequality, 
\begin{align*}
\Big\| \frac{\phi^{(1)}}{u^{\frac{3}{2}}} \Big\|_{L^2_\psi(\xi \lesssim 1)} \lesssim & \langle x \rangle^{\frac{3}{8}} \Big\| \frac{\phi^{(1)}}{\psi^{\frac{3}{4}}} \Big\|_{L^2_\psi(\xi \lesssim 1)}  \\
\lesssim & \langle x \rangle^{\frac{3}{8}} \Big[ \| \psi^{\frac{1}{4}} \phi^{(1)}_\psi \|_{L^2_\psi(\xi \lesssim 1)} + \| \phi^{(1)} \langle x \rangle^{-\frac{3}{8}} \|_{L^2_\psi(\xi \sim 1)} \Big] \\
\lesssim & \langle x \rangle^{\frac{1}{2}} \| \sqrt{\bar{u}} \phi^{(1)}_\psi \|_{L^2_\psi(\xi \lesssim 1)} + \| \phi^{(1)} \|_{L^2_\psi(\xi \sim 1)}. 
\end{align*}

The sixth term, (\ref{H1.eqn}.6), is almost exactly analogous to the fifth term which we just treated. The only exception is the control of the nonlinear part of $A_x$, which we expand here: 
\begin{align} \label{exp.A.1}
A_x = \frac{\bar{u}_{yyx}}{\bar{u}(\bar{u} + u)} + \frac{\bar{u}_{yy}}{\bar{u}(\bar{u} + u)} \frac{4 \bar{u} \bar{u}_x + \bar{u} \rho_x}{\bar{u}(\bar{u} + u)}.
\end{align}

\noindent We thus estimate the nonlinear term: 
\begin{align*}
\Big\| \frac{\rho_x}{u} \Big\|_{L^\infty_\psi} \lesssim \Big\| \frac{\phi^{(1)}}{u^2} \Big\|_{L^\infty_\psi} \lesssim \langle x \rangle^{-(\frac{5}{4}-)} E_{K_0}(x). 
\end{align*}

This concludes the proof. 

\end{proof}

\begin{lemma}[Division Estimate] \label{L.D.E} Let $\phi$ be a solution to (\ref{eps.nonlin}). Then the following estimate is valid: 
\begin{align} \n
\frac{\p_x}{2} \int |\phi^{(1)}|^2 \frac{\langle \psi \rangle}{u} \langle x \rangle^{2-2\sigma_1} +& \int |\phi^{(1)}_\psi|^2 \langle \psi \rangle \langle x \rangle^{2-2\sigma_1} \\ \n
 \le & \delta \langle x \rangle^{-(1+)} \| \phi^{(1)} \frac{\sqrt{\langle \psi \rangle}}{\sqrt{u}} \langle x \rangle^{1-\sigma} \|_{L^2_\psi}^2 + C_\delta \langle x \rangle^{-(1+)} E_0(x) \\ \label{div.mid}
 & + \eps \Big[ \langle x \rangle^{-(1+)} E_{K_0}(x) + \langle x \rangle^{-(0+)}I_{K_0}(x) \Big].
\end{align}
\end{lemma}
\begin{proof} We take inner product of (\ref{H1.eqn}) with $\phi^{(1)} \frac{\langle \psi \rangle}{u} \langle x \rangle^{2-2\sigma_1}$. The first three terms from (\ref{H1.eqn}) work in the same manner as Lemma \ref{division}, while the fourth is a further positive contribution. 

We thus estimate (\ref{H1.eqn}.5).
\begin{align} \n
&|- \int \frac{u^{(1)}}{u} A \phi \phi^{(1)} \frac{\langle \psi \rangle}{u} \langle x \rangle^{2-2\sigma_1}| \\
&\lesssim \langle x \rangle^{-(1+)} \Big\| \frac{u^{(1)}}{u} \langle x \rangle \Big\|_{L^\infty_\psi} \| A \langle x \rangle \|_{L^\infty_\psi} \Big\| \frac{\phi}{\sqrt{u}} \sqrt{\langle \psi \rangle}  \|_{L^2_\psi} \| \phi^{(1)} \frac{\langle \psi \rangle}{ \sqrt{u}} \langle x \rangle^{1-\sigma_1} \|_{L^2_\psi} \\ \n
&\lesssim C_\delta \langle x \rangle^{-(1+)} \| \phi \|_{X_1}^2 + \delta \langle x \rangle^{-(1+)} \| \phi^{(1)} \frac{\sqrt{\langle \psi \rangle}}{\sqrt{u}} \langle x \rangle^{1-\sigma_1} \|_{L^2_\psi}^2 + \eps \langle x \rangle^{-(1+)} E_{K_0}(x).
\end{align}

\noindent We have used estimate (\ref{est.A}) for $\|A\|_{L^\infty_\psi}$ and (\ref{unif.0}.5) for $\|\frac{u^{(j)}}{u} \|_{L^\infty_\psi}$.

The sixth term, (\ref{H1.eqn}.6), works almost exactly analogously to the previous term. This concludes the proof. 
\end{proof}

A nearly identical sequence of estimates is performed for the $2$ through $K_1$ order of $\p_x$: 
\begin{lemma}[Energy Estimate] Let $2 \le k \le K_1 << K_0$, and let $\phi$ solve the equation (\ref{eps.nonlin}). Let $0 < \delta << 1$ be arbitrary. Then the following inequality is valid: 
\begin{align} \n
\frac{\p_x}{2} \int |\phi^{(k)}|^2 \langle x \rangle^{2(k-\sigma_k)} &+ \int u |\phi^{(k)}_\psi|^2 \langle x \rangle^{2(k-\sigma_k)} \\ \n
\lesssim &\eps \langle x \rangle^{-(1+)} E_{K_0}(x) + \delta \langle x \rangle^{-(1+)} \| \phi^{(k)} \langle x \rangle^{k-\sigma_k} \|_{L^2_\psi}^2 \\ \n
& + C_\delta \langle  x \rangle^{-(1+)} E_{ k - 1}(x) + C_\delta \langle x \rangle^{-(0+)} I_{ k - 1}(x). 
\end{align}
\end{lemma}

\begin{lemma}[Division Estimate] Let $2 \le k \le K_1 << K_0$, and let $\phi$ solve the equation (\ref{eps.nonlin}). Let $0 < \delta << 1$ be arbitrary. Then the following inequality is valid: 
\begin{align} \n
\frac{\p_x}{2} \int |\phi^{(k)}|^2 \frac{\langle \psi \rangle}{u} \langle x \rangle^{2(k-\sigma_k)} &+ \int |\phi^{(k)}_\psi|^2 \langle \psi \rangle \langle x \rangle^{2(k-\sigma_k)} \\ \n
\lesssim  \delta \langle x \rangle^{-(1+)} E_k(x) &+ C_\delta \langle x \rangle^{-(1+)} E_{ k - 1}(x) + C_\delta \langle x \rangle^{-(0+)}I_{k-1}(x) \\ \n
& + \eps \langle x \rangle^{-(1+)} E_{K_0}(x) + \eps \langle x \rangle^{-(0+)} I_{K_0}(x). 
\end{align}
\end{lemma}

The proofs of these lemmas are essentially identical to the proofs of Lemmas \ref{L.E.E} and \ref{L.D.E}, so we omit them. 

\section{Highest Tier: $H^k$ for $K_1 < k \le K_0$}

We take $\p_x^k$ to the equation (\ref{eps.nonlin}) to obtain  
\begin{align} \n
\p_x^k \phi_x - u \p_x^k \phi_{\psi \psi} + A \p_x^k \phi - \sum_{j = 1}^k c_j \p_x^j u \p_x^{k-j} \phi_{\psi \psi} + \sum_{j = 1}^k c_j \p_x^j A \p_x^{k-j} \phi = 0.
\end{align}

\noindent We will simplify notations by setting $\phi^{(k)} := \p_x^k \phi$, in which case the above equation reads 
\begin{align} \label{dx.k}
\phi^{(k)}_x - u  \phi^{(k)}_{\psi \psi} + A  \phi^{(k)} - \sum_{j = 1}^k c_j u^{(j)}  \phi^{(k-j)}_{\psi \psi} + \sum_{j = 1}^k c_j \p_x^j A \phi^{(k-j)} = 0.
\end{align}

\begin{lemma}[Energy Estimate] Let $\phi$ be a solution to (\ref{eps.nonlin}). Then the following estimate is valid: 
\begin{align} \label{energy.k}
\frac{\p_x}{2} \int |\phi^{(k)}|^2 \langle x \rangle^{2(k - \sigma_k)} &+ \int u |\phi^{(k)}_\psi|^2 \langle x \rangle^{2(k-\sigma_k)} \\ \n
\lesssim & \delta \langle x \rangle^{-(1+)} \| \phi^{(k)} \langle x \rangle^{k-\sigma_k} \|_{L^2_\psi}^2 + C_\delta \langle x \rangle^{-(1+)} \sum_{j = 0}^{k-1} E_{j}(x) \\ \n
& + C_\delta \langle x \rangle^{-(0+)} \sum_{j = 0}^{k-1}  I_j(x)
\end{align} 
\end{lemma}
\begin{proof}  We apply the weighted multiplier $\phi^{(k)} \langle x \rangle^{2(l-\omega_l)}$ for $l = 0, ..., k$ to equation (\ref{dx.k}). Again, we write only the $l = k$ case, with the $l < k$ cases being carried out by the induction on weights argument as in Lemma \ref{L.E.E}. 

The first three terms from (\ref{dx.k}) are estimated nearly identically to the lower order energy estimates. The only difference is the following perturbative term for which we integrate by parts:
\begin{align} \n
&\int \eps \rho \phi^{(k)}_{\psi \psi} \phi^{(k)} \langle x \rangle^{2(k-\sigma_k)} \\ \label{energy.low}
= & - \int \eps \rho_\psi \phi^{(k)}_\psi \phi^{(k)} \langle x \rangle^{2(k-\sigma_k)} - \int \eps \rho |\phi^{(k)}_\psi|^2 \langle x \rangle^{2(k-\sigma_k)}. 
\end{align}

\noindent The first term above is majorized by 
\begin{align*}
|(\ref{energy.low}.1)| \lesssim &\eps \| u \rho_\psi \langle x \rangle^{\frac{3}{4}-} \|_{L^\infty_\psi} \| \phi^{(k)}_\psi \langle x \rangle^{k - \sigma_k} \|_{L^2_\psi} \Big\| \frac{\phi^{(k)}}{u} \langle x \rangle^{k-\sigma_k} \Big\|_{L^2_\psi} \\
\lesssim &\eps \langle x \rangle^{-(\frac{3}{4}-)} \| \phi_\psi \langle x \rangle^{\frac{3}{4}-} \|_{L^\infty_\psi} \| \phi^{(k)}_\psi \langle x \rangle^{k-\sigma_k} \|_{L^2_\psi} \| \sqrt{u} \phi^{(k)}_\psi \langle x \rangle^{k-\sigma_k} \|_{L^2_\psi} \langle x \rangle^{\frac{1}{2}-} \\
\lesssim & \eps \langle x \rangle^{-(\frac{1}{4}-)} E_{K_0}(x) I_{K_0}(x). 
\end{align*}

\noindent We have appealed to estimate (\ref{unif.1}.1), and the definition of the $X$ norm in (\ref{norm.X}).

The second term above is easily majorized by 
\begin{align*}
|(\ref{energy.low}.2)| \lesssim&\eps \langle x \rangle^{-(\frac{1}{4}-)} \Big\| \frac{\rho}{u} \langle x \rangle^{\frac{1}{4}-} \Big\|_{L^\infty_\psi} \| \sqrt{u} \phi^{(k)}_\psi \langle x \rangle^{k-\sigma_k} \|_{L^2_\psi}^2 \\
\lesssim & \eps \langle x \rangle^{-(\frac{1}{4}-)} E_{K_0} I_{K_0}^2. 
\end{align*}

\noindent Above, we have appealed to estimate (\ref{unif.0}.4). 

For the next terms from (\ref{dx.k}), we begin by considering the case when $j = \text{min} \{ j, k-j \}$. Consider the following term, which we integrate by parts in $\psi$:
\begin{align} \n
&- \int u^{(j)} \phi^{(k-j)}_{\psi \psi} \phi^{(k)} \langle x \rangle^{2(k-\sigma_k)} \\ \n
= &\int u^{(j)}_\psi \phi^{(k-j)}_\psi \phi^{(k)} \langle x \rangle^{2(k-\sigma_k)} + \int u^{(j)} \phi^{(k-j)}_\psi \phi^{(k)}_\psi \langle x \rangle^{2(k-\sigma_k)} \\ \n
= & \int u^{(j)}_\psi \phi^{(k-j)}_\psi \phi^{(k)} \langle x \rangle^{2(k-\sigma_k)} \chi(\xi) + \int u^{(j)}_\psi \phi^{(k-j)}_\psi \phi^{(k)} \langle x \rangle^{2(k-\sigma_k)} \chi(\xi)^C \\ \label{commute}
& +  \int u^{(j)} \phi^{(k-j)}_\psi \phi^{(k)}_\psi \langle x \rangle^{2(k-\sigma_k)} \chi(\xi) +  \int u^{(j)} \phi^{(k-j)}_\psi \phi^{(k)}_\psi \langle x \rangle^{2(k-\sigma_k)} \chi(\xi)^C. 
\end{align}

\noindent Above, $\chi(\xi)$ is a normalized cut-off function, equal to $1$ on $\xi \le 1$ and equal to $0$ on $\xi \ge 2$. We use the notation $\chi(\xi)^C := 1 - \chi(\xi)$. 

We now estimate the terms appearing above. First, 
\begin{align*}
|(\ref{commute}.1)| = & |\int \sqrt{\psi} u^{(j)}_\psi \psi^{\frac{1}{4}} \phi^{(k-j)}_\psi \frac{\phi^{(k)}}{\psi^{\frac{3}{4}}} \langle x \rangle^{2(k-\sigma_k)} \chi(\xi)| \\
\lesssim & \langle x \rangle^{\frac{3}{8}} |\int \sqrt{\xi} u^{(j)}_\psi \xi^{\frac{1}{4}} \phi^{(k-j)}_\psi \frac{\phi^{(k)}}{\psi^{\frac{3}{4}}} \langle x \rangle^{2(k- \sigma_k)}  \chi(\xi) | \\
\lesssim & \langle x \rangle^{\frac{3}{8}+ 2(k-\sigma_k)} \| u u_\psi^{(j)} \|_{L^\infty_\psi} \| \sqrt{u} \phi^{(k-j)}_\psi \|_{L^2_\psi} \Big\| \frac{\phi^{(k)}}{\psi^{\frac{3}{4}}} \Big\|_{L^2_\psi(\xi \lesssim 1)} \\
\lesssim & \langle x \rangle^{\frac{3}{8}+2(k-\sigma_k)} \langle x \rangle^{-j-\frac{1}{2}} \langle x \rangle^{-((k-j)-\sigma_{k-j})} I_{k-j}(x) E_j(x) \times \\
& \Big( \| \psi^{\frac{1}{4}} \phi^{(k)}_\psi \|_{L^2_\psi(\xi \lesssim 1)} + \| \langle x \rangle^{-\frac{3}{8}} \phi^{(k)} \|_{L^2_\psi(\xi \sim 1)} \Big) \\
\lesssim & \langle x \rangle^{\frac{3}{8}+2(k - \sigma_k)} \langle x \rangle^{-j-\frac{1}{2}} \langle x \rangle^{-(k-j) + \sigma_{k-j}} I_{k-j}(x)E_j(x) \langle x \rangle^{\frac{1}{8}} \times\\
&  \Big( \| \sqrt{u} \phi^{(k)}_\psi \|_{L^2_\psi(\xi \lesssim 1)} + \| \langle x \rangle^{-\frac{1}{2}} \phi^{(k)} \|_{L^2_\psi(\xi \sim 1)} \Big) \\
\lesssim & \delta \Big( \langle x \rangle^{-(1+)} E_k(x) + \langle x \rangle^{-(0+)} I_k(x) \Big) \\
& + C_\delta  \Big( \langle x \rangle^{-(1+)} E_{\langle k- 1 \rangle}(x) + \langle x \rangle^{-(0+)} I_{\langle k - 1 \rangle}(x) \Big) 
\end{align*}

\noindent Above, we have used $u u^{(j)}_\psi \sim u^{(j)}_y$ according to the chain rule, and subsequently (\ref{unif.1}.2).  Note that $j \le K_0 - 2$ as required by (\ref{unif.1}.2) for $k$ sufficiently large because $j = \min \{j, k - j\}$. We have also used the following Hardy type inequality:
\begin{align*}
\Big\| \frac{\phi^{(k)}}{\psi^{\frac{3}{4}}} \Big\|_{L^2_\psi(\xi \lesssim 1)} \lesssim \| \psi^{\frac{1}{4}} \phi^{(k)}_\psi \|_{L^2_\psi(\xi \lesssim 1)} + \| \phi^{(k)} \langle x \rangle^{-\frac{3}{8}} \|_{L^2(\xi \sim 1)}. 
\end{align*} 

\noindent Note that we have also used that $E_j(x) I_{k-j}(x) = I_{k-j}(x)$ by definition, and also that $k-j \le k-1$ because $j \ge 1$ and $j = \min \{ j, k-j\} \le k-1$. 

Second, in the region where $\chi(\xi)^C$ is supported, $\xi \gtrsim 1$ and so $u \gtrsim 1$. We are thus free to add in factor of $u$ which we do via:
\begin{align*}
|(\ref{commute}.2)| \lesssim & \langle x \rangle^{2(k-\sigma_k)} \langle x \rangle^{-j-\frac{1}{2}} \| u u^{(j)}_\psi \langle x \rangle^{j+\frac{1}{2}} \|_{L^\infty_\psi} \langle x \rangle^{-(k-j) + \sigma_{k-j}} \\
& \| \sqrt{u} \phi^{(k-j)}_\psi \langle x \rangle^{k-j - \sigma_{k-j}} \|_{L^2_\psi} \| \phi^{(k)} \langle x \rangle^{k - \sigma_k} \|_{L^2_\psi} \langle x \rangle^{-k+\sigma_k} \\
\lesssim & \langle x \rangle^{-(\frac{1}{2}+)} \| \sqrt{u} \phi_\psi^{(k-j)} \langle x \rangle^{k-j-\sigma_{k-j}} \|_{L^2_\psi} \| \phi^{(k)} \langle x \rangle^{k-\sigma_k} \|_{L^2_\psi} E_j(x) \\
\lesssim &  \delta \Big( \langle x \rangle^{-(1+)} E_k(x) + \langle x \rangle^{-(0+)} I_k(x) \Big) \\
& + C_\delta  \Big( \langle x \rangle^{-(1+)} E_{\langle k- 1 \rangle}(x) + \langle x \rangle^{-(0+)} I_{\langle k - 1 \rangle}(x) \Big).
\end{align*}

\noindent We have invoked the same estimates as in (\ref{commute}.2), again admissible as $j \le K_0 - 2$. 

Third, 
\begin{align*}
|(\ref{commute}.3)| = & |\int \frac{u^{(j)}}{\sqrt{\psi}} \psi^{\frac{1}{4}}\phi^{(k-j)}_\psi \psi^{\frac{1}{4}} \phi^{(k)}_\psi \langle x \rangle^{2(k-\sigma_k)} \chi(\xi)| \\
= & |\int \frac{u^{(j)}}{\sqrt{\psi}} \sqrt{\eta} \phi^{(k-j)}_\psi \sqrt{\eta} \phi^{(k)}_\psi \langle x \rangle^{2(k-\sigma_k)} \langle x \rangle^{\frac{1}{4}} \chi(\xi)| \\
\lesssim & \langle x \rangle^{-(0+)} \| \frac{u^{(j)}}{\sqrt{\psi}} \langle x \rangle^{j+ \frac{1}{4}} \|_{L^\infty_\psi} \| \sqrt{\bar{u}} \phi^{(k-j)}_\psi \langle x \rangle^{(k-j)-\sigma_{k-j}} \|_{L^2_\psi} \\
& \times \| \sqrt{\bar{u}} \phi^{(k)}_\psi \langle x \rangle^{k-\sigma_k} \|_{L^2_\psi} \\
\lesssim & \langle x \rangle^{-(0+)} \| \frac{u^{(j)}}{\eta} \langle x \rangle^{j} \|_{L^\infty_\psi}  \| \sqrt{\bar{u}} \phi^{(k-j)}_\psi \langle x \rangle^{(k-j)-\sigma_{k-j}} \|_{L^2_\psi} \\
& \times \| \sqrt{\bar{u}} \phi^{(k)}_\psi \langle x \rangle^{k-\sigma_k} \|_{L^2_\psi} \\
\lesssim &  \delta \Big( \langle x \rangle^{-(1+)} E_k(x) + \langle x \rangle^{-(0+)} I_k(x) \Big) \\
& + C_\delta  \Big( \langle x \rangle^{-(1+)} E_{\langle k- 1 \rangle}(x) + \langle x \rangle^{-(0+)} I_{\langle k - 1 \rangle}(x) \Big).
\end{align*}

\noindent Above, we have invoked estimate (\ref{unif.0}.5).

Fourth, 
\begin{align*}
|(\ref{commute}.4)| \lesssim & \langle x \rangle^{-(0+)} \| u^{(j)} \langle x \rangle^j \|_{L^\infty_\psi} \| \sqrt{\bar{u}} \phi^{(k-j)}_\psi \langle x \rangle^{(k-j) - \sigma_{k-j}} \|_{L^2_\psi} \\
\times & \| \sqrt{\bar{u}} \phi^{(k)}_\psi \langle x \rangle^{k - \sigma_k} \|_{L^2_\psi} \\
\lesssim &  \delta \Big( \langle x \rangle^{-(1+)} E_k(x) + \langle x \rangle^{-(0+)} I_k(x) \Big) \\
& + C_\delta  \Big( \langle x \rangle^{-(1+)} E_{\langle k- 1 \rangle}(x) + \langle x \rangle^{-(0+)} I_{\langle k - 1 \rangle}(x) \Big).
\end{align*}

\noindent This concludes the treatment of this term for $1 \le j = \min \{j, k-j \}$. We now treat this term for $j = \max \{j, k - j \}$. Instead of integrating by parts, we may treat the following: 
\begin{align} \label{large}
- \int u^{(j)} \phi^{(k-j)}_{\psi \psi} \phi^{(k)} \langle x \rangle^{2(k-\sigma_k)} \Big( \chi(\xi) + \chi(\xi)^C \Big). 
\end{align}

The far-field term is estimated: 
\begin{align*}
|(\ref{large}.2)| \lesssim & \langle x \rangle^{-(1+)} \| \bar{u} u^{(j)} \langle x \rangle^{j-\frac{1}{4}} \|_{L^2_\psi} \\
& \times \| \bar{u} \phi^{(k-j)}_{\psi \psi} \langle x \rangle^{(k-j)+1 - \sigma_{k-j+1}+\frac{1}{4}} \|_{L^\infty_\psi} \| \phi^{(k)} \langle x \rangle^{k - \sigma_k} \|_{L^2_\psi} \\
\lesssim & \langle x \rangle^{-(1+)} C_\delta E_{k-1}(x) + \delta \langle x \rangle^{-(1+)} \| \phi^{(k)} \langle x \rangle^{k-\sigma_k} \|_{L^2_\psi}^2 + \eps \langle x \rangle^{-(1+)} E_{K_0}(x). 
\end{align*}

\noindent Above, we use (\ref{L2.emb}.6), which is admissible since $l$ can be the top order, $K_0$ in (\ref{L2.emb}) (to deal with $j = k$ here). We also use (\ref{unif.2}.1), admissible because $k-j < K_0 - 2$. 

We estimate the localized contribution via 
\begin{align*}
|(\ref{large}.1)| =& |\int \sqrt{\psi} u^{(j)} \psi^{\frac{1}{4}} \phi^{(k-j)}_{\psi \psi} \frac{\phi^{(k)}}{\psi^{\frac{3}{4}}} \langle x \rangle^{2(k-\sigma_k)} \chi(\xi)| \\
\lesssim & \langle x \rangle^{\frac{1}{2}} \langle x \rangle^{-(\frac{5}{4}-)} \| \bar{u} u^{(j)} \langle x \rangle^{j - \frac{1}{4}} \|_{L^2_\psi} \| \sqrt{\bar{u}} \phi^{(k-j)}_{\psi \psi}  \langle x \rangle^{(k-j) + (\frac{3}{2}-)}\|_{L^\infty_\psi(\xi \lesssim 1)} \times \\
&  \| \sqrt{\bar{u}} \phi^{(k)}_\psi \langle x \rangle^{k-\sigma_k} \|_{L^2_\psi} \\
\lesssim & C_\delta \langle x \rangle^{-(1+)} E_{k-1}(x) + \delta \| \sqrt{\bar{u}} \phi^{(k)}_\psi \langle x \rangle^{k-\sigma_k} \|_{L^2_\psi}^2 \\
\lesssim & \langle x \rangle^{-(1+)} C_\delta E_{k-1}(x) + \delta \langle x \rangle^{-(1+)} \| \phi^{(k)} \langle x \rangle^{k-\sigma_k} \|_{L^2_\psi}^2 + \eps \langle x \rangle^{-(1+)} E_{K_0}(x). 
\end{align*}

\noindent Above, we have used that $u \sim \eta \sim \sqrt{\xi}$ in the region where $\xi \lesssim 1$. We have used (\ref{L2.emb}.6), again admissible since $l$ can be equal to the top order in (\ref{L2.emb}.6) (to match $j = k$ case here). We have also used the enhanced $L^\infty$ decay, (\ref{unif.2}.2), in turn admissible because $k -j < K_0 - 3$. 

We now move to the final term:  
\begin{align*}
&|\int \p_x^j A \phi^{(k-j)} \phi^{(k)} \langle x \rangle^{2(k-\sigma_k)}| \\
\le &|\int \p_x^j A \phi^{(k-j)} \phi^{(k)} \langle x \rangle^{2(k-\sigma_k)} \chi(\xi)| + |\int \p_x^j A \phi^{(k-j)} \phi^{(k)} \langle x \rangle^{2(k-\sigma_k)} \chi(\xi)^C|.
\end{align*}

\noindent We must again split into the cases when $j = \min\{j, k- j\}$ and when $j = \max\{j, k-j\}$. This is largely analogous to the previous term, and so we treat the most difficult case which is when $j = \max\{j, k-j\}$ and $\xi \lesssim 1$. First, we identify the most singular term in $\p_x^j A$ as 
\begin{align*}
&|\int \frac{\bar{u} \p_x^j u}{ \bar{u}^2(u + \bar{u})^2} \bar{u}_{yy} \phi^{(k-j)} \phi^{(k)} \langle x \rangle^{2(k-\sigma_k)} \chi(\xi)| \\
\lesssim & \langle x \rangle^{\frac{1}{2}} \| \bar{u} \p_x^j u \langle x \rangle^{j-\frac{1}{4}} \|_{L^2_\psi} \| \frac{\bar{u}_{yy}}{\bar{u}^2} \langle x \rangle \|_{L^\infty_\psi} \| \sqrt{\bar{u}} \phi^{(k-j)}_{\psi} \langle x \rangle^{(k-j) + (\frac{3}{4}-)} \|_{L^\infty_\psi} \\
& \times \| \phi^{(k)} \langle x \rangle^{k-\sigma_k} \|_{L^2_\psi} \\
\lesssim & \langle x \rangle^{-(1+)}  \| \bar{u} \p_x^j u \langle x \rangle^{j-\frac{1}{4}} \|_{L^2_\psi}\| \phi^{(k)} \langle x \rangle^{k-\sigma_k} \|_{L^2_\psi} \\
\lesssim & \delta \Big( \langle x \rangle^{-(1+)} E_k(x) + \langle x \rangle^{-(0+)} I_k(x) \Big) \\
& + C_\delta  \Big( \langle x \rangle^{-(1+)} E_{\langle k- 1 \rangle}(x) + \langle x \rangle^{-(0+)} I_{\langle k - 1 \rangle}(x) \Big). 
\end{align*}

This concludes the proof. 

\end{proof}

\begin{lemma}[Division Estimate] Let $\phi$ be a solution to (\ref{eps.nonlin}). Then the following estimate is valid: 
\begin{align} \n
\frac{\p_x}{2} \int& |\phi^{(k)}|^2 \frac{1}{u} \langle \psi \rangle \langle x \rangle^{2(k - \sigma_k)} + \int |\phi^{(k)}_\psi|^2 \langle \psi \rangle \langle x \rangle^{2(k-\sigma_k)} \\ \n
&\lesssim \delta \langle x \rangle^{-(1+)} \| \phi^{(k)} \frac{\sqrt{\langle \psi \rangle}}{\sqrt{u}} \langle x \rangle^{k-\sigma_k} \|_{L^2_\psi}^2 + C_\delta \langle x \rangle^{-(1+)} \sum_{j = 0}^{k-1} E_{j}(x) \\ \label{div.high}
& + C_\delta \langle x \rangle^{-(0+)} \sum_{j = 0}^{k-1}  I_j(x) + \eps^3 \Big[ \langle x \rangle^{-(0+)} I_{K_0}(x) + \langle x \rangle^{-(1+)} E_{K_0}(x) \Big].
\end{align}
\end{lemma}
\begin{proof} We apply the inductively weighted ``division-multiplier" $\phi^{(k)} \frac{1}{u} \langle \psi \rangle \langle x \rangle^{2(l - \omega^k_l)}$. The proof follows in essentially the same way as the baseline Division Estimate, and we thus treat the new commutator terms. First, consider 
\begin{align} \label{tcoffee}
-\int u^{(j)} \phi^{(k-j)}_{\psi \psi} \phi^{(k)} \frac{\langle \psi \rangle}{u} \langle x \rangle^{2(k-\sigma_k)} \Big( \chi(\xi) + \chi(\xi)^C \Big). 
\end{align}

We estimate the $\chi(\xi)$ term, the other being straightforward. First, we will treat the $j = 1$ case. In this case, one takes $\p_x^{k-1}$ of equation (\ref{nonlin.eq}) to generate the equality 
\begin{align*}
u \phi^{(k-1)}_{\psi \psi} = \phi^{(k)} - \underbrace{ \sum_{j = 1}^{k-1} c_j \p_x^j u \p_x^{k-j} \phi_{\psi \psi} + \sum_{j = 0}^{k-1} c_j \p_x^j A \p_x^{k-j} \phi }_{L^{(k-1)}}. 
\end{align*}

\noindent Of these terms, the highest order is $\phi^{(k)}$, which, when inserted into (\ref{tcoffee}) produces the positive term (we do not need to consider $\chi(\xi)$ here)
\begin{align*}
-\int \frac{u^{(1)}}{u} |\phi^{(k)}|^2  \frac{\langle \psi \rangle}{u} \langle x \rangle^{2(k-\sigma_k)} > 0. 
\end{align*}

\noindent The lower order terms, $L^{(k-1)}$ are inserted into (\ref{tcoffee}) to produce 
\begin{align*}
&|\int \frac{u^{(1)}}{u} L^{(k-1)} \phi^{(k)} \frac{\langle \psi \rangle}{u} \langle x \rangle^{2(k-\sigma_k)} \chi(\xi)| \\
\lesssim& \langle x \rangle^{-(1+)} \| \frac{u^{(1)}}{u} \langle x \rangle \|_{L^\infty_\psi} \Big\| L^{(k-1)} \frac{\sqrt{\langle \psi \rangle}}{\sqrt{u}} \langle x \rangle^{k-\sigma_k} \Big\|_{L^2_\psi} \| \phi^{(k)} \frac{\sqrt{\langle \psi \rangle}}{\sqrt{u}} \langle x \rangle^{k-} \|_{L^2_\psi} \\
\lesssim & \langle x \rangle^{-(1+)} \Big( C_\delta E_{k-1}(x) + \delta E_k(x) \Big). 
\end{align*}

Next, assume $2 \le j = \min \{j, k-j \}$. Then, 
\begin{align*}
|(\ref{tcoffee}.1)| \lesssim  & \langle x \rangle^{-(1+)}\Big\| \frac{u^{(j)}}{u} \langle x \rangle^j \Big\|_{L^\infty_\psi} \| \phi^{(k-j)}_{\psi \psi} \sqrt{u} \sqrt{\langle \psi \rangle} \langle x \rangle^{(k-j) + 1 - \sigma_{k-j+1}} \|_{L^2_\psi} \times \\
& \| \frac{\phi^{(k)}}{\sqrt{u}} \sqrt{ \langle \psi \rangle} \langle x \rangle^{k - \sigma_k} \|_{L^2_\psi} \\
\lesssim & \langle x \rangle^{-(1+)} E_j(x) E_{k-1}(x) E_k(x) \\
\lesssim & C_\delta \langle x \rangle^{-(1+)} E_{k-1}(x)^2 + \delta \langle x \rangle^{-(1+)} E_k(x)^2. 
\end{align*}

\noindent Above, we have used that $j \ge 2$ so that $k-j \le K_0 - 2$. We have subsequently applied (\ref{L2.emb}.11). 

Second, assume $j = \max \{j, k- j \}$. In this case, we estimate the localized contribution via 
\begin{align*}
|(\ref{tcoffee}.1)| \lesssim & \langle x \rangle^{-(1+)} \Big\| \sqrt{u} \rho^{(j)} \langle x \rangle^{j-} \sqrt{\langle \psi \rangle} \Big\|_{L^2_\psi} \| \sqrt{\bar{u}} \phi^{(k-j)}_{\psi \psi}  \langle x \rangle^{(k-j) + (\frac{5}{4}-)}\|_{L^\infty_\psi}\\
& \times \| \phi^{(k)}_\psi \sqrt{ \langle \psi \rangle} \langle x \rangle^{k - \sigma_k} \|_{L^2_\psi}  + \langle x \rangle^{-(1+)} \| \sqrt{u} \bar{u}^{(j)} \langle x \rangle^{j-} \|_{L^\infty_\psi} \\
& \times  \| \sqrt{\bar{u}} \phi^{(k-j)}_{\psi \psi} \langle x \rangle^{k-j + 1} \|_{L^2_\psi} \| \| \phi^{(k)} \frac{\sqrt{\langle \psi \rangle}}{\sqrt{u}} \langle x \rangle^{k-\sigma_k} \|_{L^2_\psi} \\
\lesssim & \langle x \rangle^{-(1+)} \Big( C_\delta E_{k-1}(x) + \delta E_k(x) \Big).
\end{align*}

\noindent Above, we have used (\ref{L2.emb}.7) which crucially allows us to not lose any derivatives ($l$ can be taken equal to $K_0$ in (\ref{L2.emb}.7)), and the decay $L^\infty$ estimates, (\ref{unif.2}.1) and (\ref{unif.2}.2). 

We next treat 
\begin{align} \label{coffee}
\int \p_x^j A \phi^{(k-j)} \phi^{(k)} \frac{\langle \psi \rangle}{u} \langle x \rangle^{2(k-\sigma_k)}.
\end{align}

\noindent  For $1 \le j = \min\{j, k-j\}$, we estimate 
\begin{align*}
|(\ref{coffee}.1)| \lesssim & \langle x \rangle^{-(1+)} \| \p_x^j A \langle x \rangle^{j+1} \|_{L^\infty_\psi} \| \frac{\phi^{(k-j)}}{\sqrt{u}} \langle \sqrt{\psi} \rangle \langle x \rangle^{(k-j) - \sigma_{k-j}} \|_{L^2_\psi} \\
&  \| \frac{\phi^{(k)}}{\sqrt{u}} \langle \sqrt{\psi} \rangle \langle x \rangle^{k - \sigma_k} \|_{L^2_\psi} \\
\lesssim & \langle x \rangle^{-(1+)} (1 + \eps E_{k-1}(x)) E_{k-1}(x) E_k(x). 
\end{align*}

Next, we consider the case when $j = \max\{j, k-j\}$. For this case, we must expand $\p_x^j A$ as in (\ref{exp.A.1}). For simplicity, we treat the term containing the highest order derivative on the unknown, $\rho$, which reads 
\begin{align} \label{djA}
\p_x^j A = \frac{\bar{u}_{yy}}{\bar{u}(\bar{u} + u)} \frac{\p_x^j \rho}{\bar{u} + u} + \text{l.o.t}.
\end{align}

\noindent Inserting this into (\ref{coffee}), we obtain 
\begin{align*}
|\int \frac{\bar{u}_{yy}}{\bar{u}(\bar{u} + u)} &\frac{\p_x^j \rho}{\bar{u} + u} \phi^{(k-j)} \phi^{(k)} \frac{\langle \psi \rangle}{u} \langle x \rangle^{2(k-\sigma_k)}| \\
\lesssim & \langle x \rangle^{-(\frac{5}{4}-)} \Big\| \frac{\bar{u}_{yy}}{\bar{u}(\bar{u} + u)} \langle x \rangle \Big\|_{L^\infty_\psi} \|\sqrt{u} \rho^{(j)} \langle x \rangle^{j-} \langle \psi \rangle^{\frac{1}{2}} \|_{L^2_\psi} \\
& \times \Big\| \frac{\phi^{(k-j)}}{u^2} \langle x \rangle^{(k-j)+\frac{1}{4}-} \Big\|_{L^\infty_\psi}  \times \| \phi^{(k)} \frac{\sqrt{\langle \psi \rangle}}{\sqrt{u}} \langle x \rangle^{k - \sigma_k} \|_{L^2_\psi} \\
\lesssim & \langle x \rangle^{-(\frac{5}{4}-)} \eps E_j(x) E_{k-1}(x) E_k(x). 
\end{align*}

\noindent The lower order terms from (\ref{djA}) are treated in the same manner. This concludes the proof. 
\end{proof}

Our scheme of \textit{a-priori} estimates, (\ref{energy.1}), (\ref{div.low}), (\ref{rhs.1}), (\ref{div.mid}), (\ref{energy.k}), (\ref{div.high}) , immediately yield the following: 
\begin{proposition}[Global Existence in $X$] Given initial data $u(1,\cdot)$ such that standard compatibility conditions are satisfied, and such that $\phi := u^2 - \bar{u}^2$ is rapidly decaying at $\infty$. Then the unique global solution $u$ guaranteed by Theorem \ref{thm.Oleinik} satisfies $\phi \in X$, where $\phi = u^2 - \bar{u}^2$. 
\end{proposition}

\section{Embeddings}

The reader should recall the specification of the $X$ norm given in (\ref{norm.X}).

\begin{lemma}[$L^2$ Estimates] For $0 \le j \le K_0 - 1$, $0 \le l \le K_0$, $1 \le m \le K_0$, $1 \le n \le K_0 - 1$,
\begin{align} \n
&\| \phi^{(l)} \langle x \rangle^{l-} \|_{L^2_\psi} + \| \phi^{(l)} \langle x \rangle^{l-} \frac{\sqrt{\langle \psi \rangle}}{\sqrt{u}} \|_{L^2_\psi} + \Big\| \frac{\phi^{(j)}}{u^2} \langle x \rangle^{j-} \Big\|_{L^2_\psi}  + \| u^{(n)} \langle x \rangle^{n-\frac{1}{4}} \|_{L^2_\psi} \\ \n
& + \Big\| \frac{u^{(n)}}{u} \langle x \rangle^{n-\frac{1}{4}} \Big\|_{L^2_\psi} + \| \bar{u} u^{(l)} \langle x \rangle^{l-\frac{1}{4}} \|_{L^2_\psi} + \| \sqrt{\bar{u}} \rho^{(l)} \sqrt{\langle \psi \rangle} \langle x \rangle^{l-\frac{1}{4}} \|_{L^2_\psi} \\ \label{L2.emb}
&+ \| \phi^{(j)}_\psi \langle x \rangle^{j + \frac{1}{2}-} \|_{L^2_\psi} + \| \phi^{(j)}_\psi \langle x \rangle^{j+\frac{1}{2}-} \frac{\sqrt{\langle \psi \rangle}}{\sqrt{u}} \|_{L^2_\psi} + \| u \phi^{(j)}_{\psi \psi} \langle x \rangle^{j + 1-} \|_{L^2_\psi} \\ \n
& + \| \sqrt{u} \phi^{(j)}_{\psi \psi} \sqrt{\langle \psi \rangle} \langle x \rangle^{j+1-} \|_{L^2_\psi}  \lesssim  \| \phi \|_X. 
\end{align}
\end{lemma}
\begin{proof} The first two terms, (\ref{L2.emb}.1) and (\ref{L2.emb}.2) are part of the $X$ norm. We next move to (\ref{L2.emb}.8). By using the Agmon inequality in the $x$ direction and subsequently Hardy inequality, we obtain 
\begin{align*}
\| \phi^{(j)}_\psi \langle x \rangle^{j + \frac{1}{2}-} \|_{L^2_\psi} \lesssim & \| \phi^{(j)}_\psi \langle x \rangle^{j-} \|_{L^2_x L^2_\psi}^{\frac{1}{2}} \Big[\| \phi^{(j+1)}_{\psi} \langle x \rangle^{j + 1-} \|_{L^2_x L^2_\psi}^{\frac{1}{2}} + \| \phi^{(j)}_\psi \langle x \rangle^{j-} \|_{L^2_x L^2_\psi}^{\frac{1}{2}}\Big] \\
\lesssim & \| \phi^{(j)}_\psi \langle x \rangle^{j-} \|_{L^2_x L^2_\psi}^{\frac{1}{2}} \| \phi_{\psi}^{(j+1)} \langle x \rangle^{j+1-} \|_{L^2_x L^2_\psi}^{\frac{1}{2}} \\
\lesssim & \| \phi \|_{X},
\end{align*}

\noindent since $j \le K_0 - 1$. The same exact proof works for (\ref{L2.emb}.9).

We may now estimate the third term, (\ref{L2.emb}.3). Since $u \gtrsim 1$ on $\xi \gtrsim 1$, (\ref{L2.emb}.3) on $\xi \gtrsim 1$ follows from (\ref{L2.emb}.1). We can thus restrict to $\xi \lesssim 1$, in which case we use that $u^2\gtrsim \eta^2 \gtrsim \xi = \frac{\psi}{\sqrt{x}}$ on the region where $\xi \lesssim 1$: 
\begin{align*}
\| \frac{\phi^{(j)}}{u^2} \|_{L^2_\psi(\xi \lesssim 1)} \lesssim & \| \frac{\phi^{(j)}}{\xi} \|_{L^2_\psi(\xi \lesssim 1)} = \sqrt{x} \| \frac{\phi^{(j)}}{\psi} \|_{L^2_\psi(\xi \lesssim 1)} \\
\lesssim & \sqrt{x} \| \phi^{(j)}_\psi \|_{L^2_\psi(\xi \lesssim 1)} + \| \phi^{(j)} \chi(\xi \sim 1) \|_{L^2_\psi} \\
\lesssim & \langle x \rangle^{\frac{1}{2}} \langle x \rangle^{-j-(\frac{1}{2}-)} + \langle x \rangle^{-(j-)}. 
\end{align*}

\noindent Above, we have used the Hardy inequality in the $\psi$ direction, admissible because $\phi^{(j)}|_{\psi = 0} = 0$. We have also used (\ref{L2.emb}.8), which is the reason we must restrict $j \le K_0 - 1$. 

For (\ref{L2.emb}.4), we split $u^{(n)} = \bar{u}^{(n)} + \rho^{(n)}$. As $\bar{u}^{(n)}$ trivially satisfies this inequality (since $n \ge 1$), we must treat $\rho^{(n)}$. By using the identity $\rho = \frac{\phi}{u + \bar{u}}$, we obtain 
\begin{align*}
\rho^{(n)} = \frac{\phi^{(n)}}{u + \bar{u}} + \sum_{k < n} c_k \phi^{(k)} \p_x^{n-k} \frac{1}{u + \bar{u}},
\end{align*}

\noindent from where we obtain 
\begin{align} \label{extra}
\| \rho^{(n)} \|_{L^2_\psi} \lesssim \Big\| \frac{\phi^{(n)}}{u + \bar{u}} \Big\|_{L^2_\psi} \lesssim \Big\| \frac{\phi^{(n)}}{\bar{u}} \Big\|_{L^2_\psi} \lesssim \Big\| \frac{\phi^{(n)}}{\bar{u}^2} \Big\|_{L^2_\psi}
\end{align}

\noindent from here the result follows from the corresponding $\phi^{(n)}$ estimate, which we can use because $n \le N_0 - 1$. The same proof works for  (\ref{L2.emb}.5) upon noticing that we can put an extra factor of $\frac{1}{u}$ in (\ref{extra}).

For the term (\ref{L2.emb}.10), we have, using the equation (\ref{nonlin.eq}),
\begin{align*}
 \| u \phi_{\psi \psi} \|_{L^2_\psi} \lesssim  \| \phi_x \|_{L^2_\psi} + \| A \phi \|_{L^2_\psi} \lesssim  \langle x \rangle^{-(1-)} \| \phi \|_{X_1}.
\end{align*}

\noindent Clearly, we may upgrade to the general $j \le K_0 - 1$ case. The proof of (\ref{L2.emb}.11) works in an identical manner. This concludes the proof. 

\end{proof}

A key feature we take advantage of is that decay is enhanced in the region $\xi \lesssim 1$. Note that this type of enhanced decay is not available at the top two orders of $\p_x$ (as seen by the restriction on $j$ below).

\begin{lemma}[$L^2(\xi \lesssim 1)$ Estimates] For $\alpha = 0,1,2$, and for $0 \le j \le K_0-2$,
\begin{align} \label{L2.enhance}
\| \p_\psi^\alpha \phi^{(j)} \langle x \rangle^{j+ \frac{\alpha}{2}+ \frac{1}{4}-} \|_{L^2_\psi(\xi \lesssim 1)}  \lesssim \| \phi \|_{X}.
\end{align}
\end{lemma}
\begin{proof} We address the $j = 0$ case, the general $j$ case being analogous. We begin with rearranging (\ref{nonlin.eq}) to obtain 
\begin{align*}
\| \phi_{\psi \psi} \|_{L^2_\psi (\xi \lesssim 1)} = & \| \frac{1}{u} \phi_x \|_{L^2_\psi(\xi \lesssim 1)} + \| \frac{1}{u} A \phi \|_{L^2_\psi(\xi \lesssim 1)} \\
\lesssim & \| \frac{1}{\sqrt{\xi}} \phi_x \|_{L^2_\psi(\xi \lesssim 1)} + \| \frac{1}{\sqrt{\xi}} A \phi \|_{L^2_\psi(\xi \lesssim 1)} \\
= & \| \frac{x^{\frac{1}{4}}}{\psi^{\frac{1}{2}}} ( \phi_x + A \phi) \|_{L^2_{\psi}(\xi \lesssim 1)} \\
\lesssim & x^{\frac{1+}{4}} \| \frac{1}{\psi^{\frac{1+}{2}}} (\phi_x + A \phi) \|_{L^2_\psi(\xi \lesssim 1)} \\
\lesssim & x^{\frac{1+}{4}} \| \psi^{\frac{1-}{2}} ( \phi^{(1)}_{\psi} + \p_\psi \{ A \phi \} ) \|_{L^2_\psi(\xi \lesssim 1)} \\
\lesssim & x^{\frac{1+}{4}} \| \langle \sqrt{\psi} \rangle \phi^{(1)}_\psi \|_{L^2_\psi(\xi \lesssim 1) } + \\
\lesssim & x^{\frac{1+}{4}} \langle x \rangle^{-(1-)} \| \langle \sqrt{\psi} \rangle \phi^{(1)}_\psi \langle x \rangle^{1-\sigma_1} \|_{L^2_\psi(\xi \lesssim 1)} \\
\lesssim & \langle x \rangle^{-\frac{5-}{4}} \| \phi \|_X. 
\end{align*}

\noindent Above, we have used that, for $j \le k-1$, $\| \langle \sqrt{\psi} \rangle \phi^{(j)}_\psi \|_{L^2_\psi} \lesssim \langle x \rangle^{-j-(\frac{1}{2}-)}$ according to (\ref{L2.emb}.7). 

We localize the $\phi^{(j)}$ estimate via 
\begin{align*}
|\phi| = |\int_0^\psi \phi_{\psi}| \lesssim \sqrt{\psi} \| \phi_\psi \|_{L^2_\psi(\xi \lesssim 1)} \lesssim \langle x \rangle^{-\frac{1}{2}} \Big( \langle x \rangle^{\frac{3}{4}} \| \phi_\psi \|_{L^2_\psi(\xi \lesssim 1)} \Big),
\end{align*}

\noindent which implies 
\begin{align*}
\| \phi \|_{L^2_\psi(\xi \lesssim 1)} \lesssim \langle x \rangle^{-(\frac{1}{4}-)} \Big( \langle x \rangle^{\frac{3}{4}} \| \phi_\psi \|_{L^2_\psi(\xi \lesssim 1)} \Big).
\end{align*}

For the enhanced localized $\phi_\psi$ estimate, we have 
\begin{align*}
\| \phi_\psi \|_{L^2_\psi(\xi \lesssim 1)} \lesssim & \| \phi \|_{L^2_\psi(\xi \lesssim 1)}^{\frac{1}{2}}  \| \phi_{\psi \psi} \|_{L^2_\psi(\xi \lesssim 1)}^{\frac{1}{2}} \\
\lesssim &  \langle x \rangle^{-(\frac{1}{8}-)} \Big( \langle x \rangle^{\frac{3}{4}-} \| \phi_\psi \|_{L^2_\psi(\xi \lesssim 1)} \Big)^{\frac{1}{2}} \langle x \rangle^{-(\frac{5}{8}-)} \| \phi \|_X^{\frac{1}{2}}. 
\end{align*} 

This concludes the proof.

\end{proof}
\begin{lemma}[$L^\infty$ Estimates] Let $0 \le j \le K_0 - 2$, 
\begin{align}  \n
&\Big[ \Big\| \frac{\phi^{(j)}}{u^2} \Big\|_{L^\infty_\psi} + \Big\| \frac{\phi^{(j)}}{u} \Big\|_{L^\infty_\psi} + \| \phi^{(j)} \|_{L^\infty_\psi} + \Big\| \frac{\rho^{(j)}}{u} \Big\|_{L^\infty_\psi} \Big] \langle x \rangle^{j + (\frac{1}{4}-)} \\ \label{unif.0}
& \hspace{15 mm} + \Big\| \frac{u^{(j)}}{u} \Big\|_{L^\infty_\psi} \langle x \rangle^{j} \lesssim \| \phi \|_X, \\ \label{unif.1}
& \| \phi^{(j)}_\psi \langle x \rangle^{j+(\frac{3}{4}-)}  \|_{L^\infty_\psi} + \| u^{(j)}_y \langle x \rangle^{j + \frac{1}{2}} \|_{L^\infty_\psi}  \lesssim \| \phi \|_X, \\ \label{unif.2}
&\| \bar{u} \phi^{(j)}_{\psi \psi} \|_{L^\infty_\psi} \langle x \rangle^{j + (\frac{5}{4}-)} + \| \sqrt{\bar{u}} \phi^{(j)}_{\psi \psi} \|_{L^\infty_\psi(\xi \lesssim 1)} \langle x \rangle^{j+\frac{3}{2}-} \lesssim \| \phi \|_X. 
\end{align}
\end{lemma}
\begin{proof} We start with (\ref{unif.0}.3), for which a standard interpolation gives 
\begin{align*}
\| \phi^{(j)} \|_{L^\infty_\psi} \lesssim & \| \phi^{(j)} \|_{L^2_\psi}^{\frac{1}{2}} \| \phi^{(j)}_\psi \|_{L^2_\psi}^{\frac{1}{2}} \\
\lesssim & \Big( \langle x \rangle^{j-} \| \phi \|_{X_j} \Big)^{\frac{1}{2}} \Big( \langle x \rangle^{-j - (\frac{1}{2}-)} \| \phi^{(j)} \langle x \rangle^{j + \frac{1}{2}-} \|_{L^2_\psi} \Big)^{\frac{1}{2}},
\end{align*}

\noindent upon using (\ref{L2.emb}.1) and (\ref{L2.emb}.8).

We now estimate (\ref{unif.1}.1). We simply interpolate (integrating from $\psi = \infty$): 
\begin{align*}
\|\phi_\psi^{(j)}\|_{L^\infty_\psi} \lesssim \| \phi_\psi^{(j)} \|_{L^2_\psi}^{\frac{1}{2}} \| \phi_{\psi \psi}^{(j)} \|_{L^2_\psi}^{\frac{1}{2}} \lesssim \Big( \langle x \rangle^{-j-(\frac{1}{2}-)} \Big)^{\frac{1}{2}} \Big(\langle x \rangle^{-j-(1-)} \Big)^{\frac{1}{2}},
\end{align*}

\noindent upon using (\ref{L2.emb}.8), (\ref{L2.emb}.10), and (\ref{L2.enhance}).

We now move to (\ref{unif.0}.1). Clearly we may restrict to the region $\xi \lesssim 1$, in which case $u^2 \sim \eta^2 \sim \xi$. Since $\phi|_{\psi = 0} = 0$, we may write 
\begin{align*}
|\frac{\phi}{u^2}| \lesssim & \frac{1}{\xi} |\phi(x,\psi)| \lesssim \frac{\sqrt{x}}{\psi} |\int_0^\psi \phi_\psi(x,\psi') \ud \psi'| \lesssim  \sqrt{x} \| \phi_\psi \|_{L^\infty_\psi} \\
\lesssim & \langle x \rangle^{\frac{1}{2}} \langle x \rangle^{-\frac{3-}{4}} = \langle x \rangle^{-\frac{1-}{4}}.
\end{align*}

\noindent The proofs of (\ref{unif.0}.2) and (\ref{unif.0}.4) are identical.  

We move to (\ref{unif.1}.2). Recall that 
\begin{align*}
\p_\psi \phi = 2 uu_y \frac{1}{u} - 2 \bar{u} \bar{u}_y \frac{1}{\bar{u}} = 2 (u_y - \bar{u}_y), 
\end{align*}

\noindent and thus the estimate follows upon noticing that it holds for both $\phi^{(j)}_\psi$ and $\bar{u}^{(j)}_y$. 

Finally, to estimate the $\phi_{\psi \psi}$ term, we use the equation. The general $j$ case follows in a similar manner, so we deal with $j = 0$: 
\begin{align*}
\| \bar{u} \phi^{(j)}_{\psi \psi} \|_{L^\infty_\psi} = \| u^{(j+1)} \|_{L^\infty_\psi} +  \| \frac{1}{x} u^{(j)} \|_{L^\infty_\psi} \lesssim \langle x \rangle^{-(j+1) - (\frac{1}{4}-)} \| \phi \|_X. 
\end{align*}

In a similar fashion, 
\begin{align*}
\| \sqrt{\bar{u}} \phi_{\psi \psi} \|_{L^\infty_\psi(\xi \lesssim 1)}  = & \| \frac{1}{\sqrt{\bar{u}}} \phi^{(1)} \|_{L^\infty_\psi(\xi \lesssim 1)} + \| \frac{A}{\sqrt{u}} \phi \|_{L^\infty_\psi(\xi \lesssim 1)}.
\end{align*}

\noindent We treat the first term, as the second term above is analogous. We estimate 
\begin{align*}
|\frac{1}{u} \phi^{(1)}| \lesssim & \langle x \rangle^{\frac{1}{8}} \psi^{-\frac{1}{4}} |\int_0^\psi \phi^{(1)}_\psi| \lesssim \langle x \rangle^{\frac{1}{8}} \psi^{-\frac{1}{4}} \psi^{\frac{1}{2}} \| \phi^{(1)}_\psi \|_{L^2_\psi(\xi \lesssim 1)} \\
\lesssim & \langle x \rangle^{\frac{1}{4}} \| \phi^{(1)}_\psi \|_{L^2_\psi(\xi \lesssim 1)} \lesssim \langle x \rangle^{-(\frac{3}{2}-)} \| \phi \|_X. 
\end{align*}

This concludes the proof. 
\end{proof}

\section{Global Positivity of $\Omega$}

We now analyze the quantity 
\begin{align*}
\Omega = - \bar{u}_{yy} + \frac{1}{2} \bar{u} \bar{u}_x.
\end{align*} 

\begin{lemma} \label{Omega.lemma} For $[\bar{u}, \bar{v}]$ Blasius solutions, $\Omega \ge 0$.  
\end{lemma}
\begin{proof} Since $\bar{u}$ is a solution to the Prandtl equation, which yields the identity 
\begin{align*}
\Omega = - \frac{1}{2} \bar{u}_{yy} - \frac{1}{2} \bar{v} \bar{u}_y. 
\end{align*}

As $[\bar{u}, \bar{v}]$ are Blasius solutions, we may further invoke the self-similar structure and subsequently the self-similar ODE satisfied by $f$ to rewrite 
\begin{align*}
2x \Omega =& - (\eta f' - f) f'' - f''' \\
= & - \eta f' f'' + f f'' + f f'' \\
= & (- \eta f' + 2f) f''. 
\end{align*}

For Blasius solutions, $f'' > 0$, and the question, therefore, reduces to establishing nonnegativity of the quantity $- \eta f' + 2f$, which we thus name $\omega$. 

First, as $f' \rightarrow 1$ and $f \rightarrow \eta$ at $\eta = \infty$, clearly 
\begin{align*}
2f > \eta f' \text{ as } \eta \rightarrow \infty.
\end{align*} 

We now analyze a sufficiently small neighborhood of $\eta = 0$. The following Taylor expansions are valid: 
\begin{align*}
&f =  f(0) + \eta f'(0) + \frac{\eta^2}{2} f''(0) + \frac{\eta^3}{6} f'''(0) + \bigO(\eta^4) \\
& \hspace{3 mm} = \frac{\eta^2}{2}f''(0) + \frac{\eta^3}{6} f'''(0) + \bigO(\eta^4) \\
&\eta f' = \eta [f'(0) + f''(0) \eta + \frac{f'''(0)}{2} \eta^2] + \bigO(\eta^4) \\
& \hspace{4 mm} = \eta^2 f''(0) + \frac{1}{2} f'''(0) \eta^3 + \bigO(\eta^4). 
\end{align*}

\noindent Multiplying the first quantity, $f$, by $2$ we see that the $f''$ terms match. The $f'''(0)$ terms have a factor of $\frac{1}{3}$ versus $\frac{1}{2}$ for $\eta f'$. In general, the $f^{(n)}(0)$ terms have a factor of $2 \frac{1}{n!}$ which is less than the factor of $\frac{1}{(n-1)!}$. Upon realizing that $f'''(0) = f^{(4)} = 0$ and then the first nonzero term is $f^{(5)}(0) < 0$, the positivity of $\omega > 0$ for $\eta << 1$ follows. In turn, this follows from differentiating the Blasius ODE to obtain 
\begin{align*}
f^{(5)} = - |f''|^2 - 2f' f''' - f f^{(4)}.  
\end{align*}

We must now analyze $w$ for $\eta$ in the ``in-between" regions. We need to check that $\omega$ cannot change sign. This would be implied if $\omega$ was non-decreasing. Taking one derivative: 
\begin{align*}
\p_\eta \omega = \p_\eta \{ 2f - \eta f' \} = f' - \eta f''. 
\end{align*}

The aim is to check the right-hand side is nonnegative. Again, we can check that this quantity is zero at $y = 0$ and $1$ at $y = \infty$, and thus would be nonnegative if it were to be monotonically increasing. Taking a further derivative, we obtain $- \eta f''' > 0$. 
\end{proof}

\section{Weighted Nash Inequality and Optimal Decay}

\begin{lemma} \label{lemma.nash} Solutions $\phi \in X$ to the system (\ref{eps.nonlin}) satisfy the following Nash-type inequality 
\begin{align}
\begin{aligned} \label{Nash.1}
\| \phi \|_{L^2_\psi}^2 \lesssim \max \begin{cases} x^{\frac{1}{10}} \| \sqrt{u} \phi_\psi \|_{L^2_\psi}^{\frac{4}{5}} \\ \| \sqrt{u} \phi_\psi \|_{L^2_\psi}^{\frac{2}{3}} \end{cases} 
\end{aligned}
\end{align}
\end{lemma}
\begin{proof} We first localize based on $\xi = \frac{\psi}{\sqrt{x}}$. Fix a $\tau$ to be selected later. Then by triangle inequality we split
\begin{align} \label{tri}
\| \phi \|_{L^2_\psi} \le \| \phi \chi(\frac{\xi}{\tau}) \|_{L^2_\psi} + \| \phi \chi(\frac{\xi}{\tau})^C \|_{L^2_\psi}
\end{align}

\noindent For the localized portion, we need to condition on whether or not $\tau < 1$ or $\tau > 1$. We integrate by parts via 
\begin{align}\label{chi}
\| \phi \chi(\frac{\xi}{\tau}) \|_{L^2_\psi}^2 = \int \p_\psi \{ \psi \} \phi^2 \chi(\frac{\xi}{\tau})^2 = - \int 2 \psi \phi \phi_\psi \chi(\frac{\xi}{\tau})^2 - \int \psi \phi^2 \frac{1}{\sqrt{x}} \frac{1}{\tau} \chi \chi'.  
\end{align}

\noindent We estimate the former term above term via 
\begin{align*}
|\int \psi \phi \phi_\psi \chi(\frac{\xi}{\tau})^2| \lesssim \begin{cases} \tau^{\frac{3}{2}}x \| \sqrt{u} \phi_\psi \|_{L^2_\psi}^2 \text{ if } \tau < 1 \\ \rho^2 x \| \sqrt{u} \phi_\psi \|_{L^2_\psi}^2 \text{ if } \tau \ge 1  \end{cases}
\end{align*}

\noindent More specifically, in the case when $\rho < 1$
\begin{align*}
|\int \psi \phi \phi_\psi \chi(\frac{\xi}{\tau})^2| \le & \| \phi \chi \|_{L_\psi^2} \| \psi \phi_\psi \chi \|_{L^2_\psi} \lesssim  \| \phi \chi \|_{L^2_\psi} \| \frac{\psi}{\sqrt{x}} \phi_\psi \chi \|_{L^2_\psi} \sqrt{x} \\
\lesssim & \| \phi \chi \|_{L^2_\psi} \sqrt{x} \| \xi \phi_\psi \chi \|_{L^2_\psi} \lesssim \| \phi \chi \| \sqrt{x} \tau^{\frac{3}{4}} \| \xi^{\frac{1}{4}} \phi_\psi \chi \|_{L^2_\psi} \\
\lesssim & \| \phi \chi \|_{L^2_\psi} \sqrt{x} \tau^{\frac{3}{4}} \| \sqrt{u} \phi_\psi \chi \|_{L^2_\psi} \\
\le & o(1) \| \phi \chi \|_{L^2_\psi}^2 + C x \tau^{\frac{3}{2}} \| \sqrt{u} \phi_\psi \chi \|_{L^2_\psi}^2. 
\end{align*}

\noindent The $o(1)$ term is absorbed to the left-hand side of (\ref{chi}). 

In the case when $\rho > 1$, we must estimate $\xi^{\frac{1}{4}} \le \tau^{\frac{1}{4}} \sqrt{u}$. To see that this is true, first assume $\xi \le 1$. Then $\xi^{\frac{1}{4}} \lesssim \sqrt{u} \lesssim \sqrt{u} \tau^{\frac{1}{4}}$ because $\tau > 1$ by assumption. Next, suppose $\xi \ge 1$. Then $\xi^{\frac{1}{4}} \le \tau^{\frac{1}{4}} \lesssim \tau^{\frac{1}{4}} \sqrt{u}$ because $u \gtrsim 1$ on the region when $\xi \ge 1$. 

For the second term in (\ref{chi}), we estimate identically to the far-field term from (\ref{tri}), which we now treat. 

For the far-field term, we estimate via 
\begin{align*}
|\int \phi^2 \psi \frac{1}{\psi} \chi(\frac{\xi}{\tau})^C| \lesssim \frac{1}{\tau \sqrt{x}} \| \phi \sqrt{\psi} \|_{L^2_\psi}^2. 
\end{align*}

In summary, we have thus established the inequality 
\begin{align*}
\| \phi \|_{L^2_\psi}^2 \lesssim \varphi(\tau) x \| \sqrt{u} \phi_\psi \|_{L^2_\psi}^2 + \frac{1}{\tau \sqrt{x}} \| \phi \sqrt{\psi} \|_{L^2_\psi}^2. 
\end{align*}

\noindent where $\varphi(\tau)$ is the piecewise function equal to $\tau^{\frac{3}{2}}$ on $\tau < 1$ and $\tau^2$ on $\rho \ge 1$. 

We now select 
\begin{align*}
\tau = \begin{cases} x^{-\frac{3}{5}} \| \phi \sqrt{\psi} \|_{L^2_\psi}^{\frac{4}{5}} \| \sqrt{\bar{u}} \phi_\psi \|_{L^2_\psi}^{-\frac{4}{5}} := r^{\frac{6}{5}} \text{ if } r < 1  \\
x^{-\frac{1}{2}} \| \phi \sqrt{\psi} \|_{L^2_\psi}^{\frac{2}{3}} \| \phi_\psi \sqrt{\bar{u}} \|_{L^2_\psi}^{-\frac{2}{3}} := r  \ge 1 \end{cases}
\end{align*}

\noindent The key point is that $\tau$ is homogeneous in $r$, and therefore we may consistently enforce when $\tau < 1$ and $\tau > 1$ because these are equivalent to $r < 1$ and $r > 1$. 

To conclude, we note that by definition of the $X$ norm, the weighted quantities $\| \phi \sqrt{\psi} \|_{L^2_\psi}$ are conserved in $x$ for solutions to (\ref{eps.nonlin}). This immediately gives (\ref{Nash.1}).
\end{proof}

We are now ready to establish the optimal decay rates.

\begin{proof}[Proof of Theorem \ref{thm.main}] Using (\ref{Nash.1}) in (\ref{energy.1}), letting $\alpha(x) := \| \phi \|_{L^2_\psi}^2$, we obtain either one of the two ODEs ($\cdot = \p_x$) : 
\begin{align*}
\dot{\alpha} + C_0 \alpha^3 \le 0 \text{ or } \dot{\alpha} + C_1 x^{-\frac{1}{4}} \alpha^{\frac{5}{2}} \le 0
\end{align*}

\noindent at each $x \in \mathbb{R}_+$. This immediately implies that $|\alpha| \lesssim \langle x \rangle^{-\frac{1}{2}}$, which means that $\| \phi \|_{L^2_\psi} \lesssim \langle x \rangle^{-\frac{1}{4}}$. 

We may $x$-differentiate (\ref{Nash.1}) and use them in the higher order energy estimates (\ref{rhs.1}), (\ref{energy.k}) in exactly the same fashion which yields $\| \phi^{(j)} \|_{L^2_\psi} \lesssim \langle x \rangle^{-j-\frac{1}{4}}$. From here, one repeats all of the above $X$-norm estimates with self-similar weights $\langle \eta \rangle$, which follows in an identical manner.  The estimates (\ref{est.main}) follow from standard interpolation, the relations 
\begin{align*}
u - \bar{u} =  \frac{\phi}{u + \bar{u}}, \hspace{2 mm} \p_y u = u \p_\psi u, \hspace{2 mm} \p_y \bar{u} = \bar{u} \p_\psi \bar{u},
\end{align*}

\noindent and Lemmas 4, 4' in \cite{Serrin} to replace the stream function variable $\psi$ with the physical variable $y$ (one takes the parameter $a$ in Lemmas $4, 4'$ of \cite{Serrin} to be $\langle x \rangle^{-\frac{1}{2}}$).  
\end{proof}

\renewcommand\refname{References}
\def\bibindent{3.5em}

\end{document}